\documentclass[11pt,reqno, a4paper]{amsart}

\usepackage{amsmath,amssymb,amsthm,enumitem,xcolor,tikz,tikz-cd,subcaption,hyperref,mathtools,float,graphicx}

\usepackage[style=alphabetic,sorting=nty,doi=false,isbn=false,url=false]{biblatex}

\pgfdeclarelayer{fg}
\pgfsetlayers{main, fg}

\usepackage{cleveref}

\tikzset{
	on each segment/.style={
		decorate,
		decoration={
			show path construction,
			moveto code={},
			lineto code={
				\path [#1]
				(\tikzinputsegmentfirst) -- (\tikzinputsegmentlast);
			},
			curveto code={
				\path [#1] (\tikzinputsegmentfirst)
				.. controls
				(\tikzinputsegmentsupporta) and (\tikzinputsegmentsupportb)
				..
				(\tikzinputsegmentlast);
			},
			closepath code={
				\path [#1]
				(\tikzinputsegmentfirst) -- (\tikzinputsegmentlast);
			},
		},
	},
	mid arrow/.style={postaction={decorate,decoration={
				markings,
				mark=at position .5 with {\arrow[#1]{stealth}}
	}}},
	mid arrowa/.style 2 args={postaction={decorate,decoration={
				markings,
				mark=at position #2 with {\arrow[#1]{stealth}}
	}}},
	mid arrowb/.style={postaction={decorate,decoration={
				markings,
				mark=at position 0.53 with {\arrow[#1]{stealth}}
	}}},
	every loop/.style={decorate,min distance=15mm,in=50,out=130,looseness=10
	},
	mid arrowa/.default={black}{0.55}
}

\usetikzlibrary{positioning,shapes,fit,arrows,shadows,matrix,arrows.meta,decorations.pathreplacing,decorations.markings}

\numberwithin{equation}{section}
\hypersetup{  pdfborder={0 0 0},
	colorlinks   = true,
	urlcolor     = blue,
	linkcolor    = blue,
	citecolor   = red}

\makeatletter
\newcommand\bigcdot{\mathpalette\bigcdot@{.5}}
\newcommand\bigcdot@[2]{\mathbin{\vcenter{\hbox{\scalebox{#2}{$\m@th#1\bullet$}}}}}
 
\makeatother

\newcommand{\N}{\mathbb{N}}
\newcommand{\Z}{\mathbb{Z}}
\newcommand{\R}{\mathbb{R}}
\newcommand{\Q}{\mathbb{Q}}

\DeclareMathOperator{\C}{C}
\DeclareMathOperator{\cent}{Z}
\DeclareMathOperator{\Out}{Out}
\DeclareMathOperator{\Aut}{Aut}

\newcommand{\vr}{ \leqslant_{vr}}
\newcommand{\nvr}{ \lhd_{vr}}
\newcommand{\n}{\vartriangleleft}

\newcommand{\VRC}{(VRC)}

\newcommand\restr[2]{{
		\left.\kern-\nulldelimiterspace 
		#1 
		\littletaller 
		\right|_{#2} 
}}

\newcommand{\littletaller}{\mathchoice{\vphantom{\big|}}{}{}{}}
\newcommand{\ov}{\overline}
\newcommand{\ot}{\widetilde}

\newtheorem{thm}{Theorem}[section]
\AddToHook{env/thm/begin}{\crefalias{thm}{thm}}

\newtheorem{prop}[thm]{Proposition}
\AddToHook{env/prop/begin}{\crefalias{thm}{prop}}

\newtheorem{lemma}[thm]{Lemma}
\AddToHook{env/lemma/begin}{\crefalias{thm}{lemma}}

\newtheorem{cor}[thm]{Corollary}
\AddToHook{env/cor/begin}{\crefalias{thm}{cor}}

\AddToHook{env/claim/begin}{\crefalias{section}{claim}}

\theoremstyle{definition}
\newtheorem{defn}[thm]{Definition}
\AddToHook{env/defn/begin}{\crefalias{thm}{defn}}

\newtheorem{question}[thm]{Question}
\AddToHook{env/question/begin}{\crefalias{thm}{question}}

\newtheorem{problem}[thm]{Problem}
\AddToHook{env/problem/begin}{\crefalias{thm}{problem}}

\AddToHook{env/constr/begin}{\crefalias{thm}{constr}}

\theoremstyle{remark}
\newtheorem{ex}[thm]{Example}
\AddToHook{env/ex/begin}{\crefalias{thm}{ex}}

\newtheorem{rem}[thm]{Remark}
\AddToHook{env/rem/begin}{\crefalias{thm}{rem}}

\AddToHook{env/exercise/begin}{\crefalias{thm}{exercise}}

\newtheorem{notation}[thm]{Notation}
\AddToHook{env/notation/begin}{\crefalias{thm}{notation}}

\newcommand{\anonym}{0} 

\title{Virtual retractions in free constructions}

\ifnum \anonym=0
\author[Merladet Urig\"uen and Minasyan]{Jon Merladet Urig\"uen and Ashot Minasyan 
\\ with an appendix by Jon Merladet Urig\"uen, Ashot Minasyan, Xiaolei Wu and Shengkui Ye
}

\address[J. Merladet and A. Minasyan]{CGTA, School of Mathematical Sciences, University of Southampton, Highfield, Southampton, SO17~1BJ, United Kingdom}
\email{J.F.Merladet@soton.ac.uk, aminasyan@gmail.com}

\address[X. Wu]{Shanghai Center for Mathematical Sciences, Jiangwan Campus, Fudan University, No.2005 Songhu Road, Shanghai, 200438, P.R. China}
\email{xiaoleiwu@fudan.edu.cn}

\address[S. Ye]{NYU Shanghai, No.567 Yangsi West 
  Rd, Pudong New Area, Shanghai, 200124, P.R. China \\
NYU-ECNU Institute of Mathematical Sciences at NYU Shanghai, 3663 Zhongshan Road North, Shanghai, 200062, China}
\email{sy55@nyu.edu}
\fi


\begin{document}
\vspace*{-1cm}	
	
\begin{abstract}
A group $G$ has property  (VRC) if every cyclic subgroup is a virtual retract. This property is stable under many standard group-theoretic constructions and is enjoyed by all virtually special groups (in the sense of Haglund and Wise). In this paper we study property (VRC) for fundamental groups of finite graphs of groups. 

Our main criterion shows that the fundamental group of a finite graph of finitely generated virtually abelian groups has (VRC) if and only if it has a homomorphism to a Euclidean-by-finite group that is injective on all vertex groups. This result allows us to determine property (VRC) for such groups using basic tools from Euclidean Geometry and Linear Algebra. We use it to produce examples and to give sufficient criteria for fundamental groups of finite graphs of finitely generated abelian groups with cyclic edge groups to have (VRC).

In the last two sections and in the appendix we give applications of property (VRC). We show that if a fundamental group of a finite graph of groups with finitely generated virtually abelian vertex groups has (VRC) then it is CAT($0$). We also show that 
tubular groups with (VRC) are virtually free-by-cyclic and virtually special.
\end{abstract}

	\maketitle

    \setcounter{tocdepth}{1}
\tableofcontents

    
\section{Introduction}
A subgroup $H$ is a \emph{virtual retract of} a group $G$ if there is a finite index subgroup $K \leqslant_f G$ such that $H$ is a \emph{retract} of $K$, that is $H \subseteq K$ and there is a homomorphism $\rho:K \to H$ whose restriction to $H$ is the identity map. Virtual retracts play an important role in Group Theory and have applications towards both algebraic and geometric properties, see \cite{virtprops} and references therein.

The group $G$ is said to have property \emph{(LR)} if every finitely generated subgroup is a virtual retract of $G$, and $G$ has property  \emph{(VRC)} if every cyclic subgroup is a virtual retract of $G$.  The study of  these properties was initiated by Long and Reid \cite{Long-Reid}, and continued by Minasyan in \cite{virtprops}. Evidently, (LR) is much stronger than (VRC). Both of these properties have important consequences for the profinite topology: a group with (LR) is necessarily subgroup separable (LERF), and a group with (VRC) is cyclic subgroup separable (and, hence, residually finite).

\subsection{Property (LR)}
Basic examples of groups with (LR) include finitely generated virtually abelian groups \cite{virtprops} and free groups of finite rank \cite{HallFreeLERF}. Scott \cite{Scott} showed that surface groups have (LR), which was generalized to all limit groups by Wilton \cite{WiltonLimitGps}. In \cite{GITIK2003} Gitik, Margolis and Steinberg proved that property (LR) is stable under free products, and in \cite{virtprops} Minasyan showed that the direct product of a group with (LR) and a finitely generated virtually abelian group has (LR). Property (LR) is also clearly preserved by taking arbitrary subgroups. In \cite{virtprops}, it was shown that for a finitely generated abelian group $A$, $A \wr \Z$ has (LR) if and only if $A$ is finite and semisimple. This more or less completes the list of known groups with (LR). 

One of the main difficulties with (LR) is that this property is not in general preserved by direct products and by taking finite index supergroups \cite{virtprops}. For instance, it is not known whether all finitely generated virtually free groups have (LR) (see \cite[Question~11.1]{virtprops}). 

In this work we develop new tools for showing that groups satisfy (LR) or (VRC), with a particular emphasis on fundamental groups of finite graphs of groups.
First, we extend the class of known (LR) groups, by showing that it is preserved under extensions with finitely generated virtually abelian kernels that are virtual retracts (see \Cref{cor:LR extension from quot}), and use it to prove the following.

\begin{cor} \label{cor:LR_for_amalgam} Let $G=A*_N B$ be an amalgamated free product of two finitely generated virtually abelian groups $A$ and $B$ over a common normal subgroup $N \n A$ and $N \n B$. Then $G$ has (LR) if and only if the natural images of $A$ and $B$ in $\Out(N)$ generate a finite subgroup. In particular, this is  the case when one of $A$ or $B$ is abelian.
\end{cor}

See \Cref{cor:(LR)_for_HNN} for a corresponding statement for HNN-extensions and \Cref{prop:common amalg norm sub} for a statement covering more general graphs of groups.

\begin{ex}\label{ex:amalgam_of_two_virt_Z^n_gps}
Let $n \in \N$, $N \coloneq \Z^n$ and let $x,y$ be two finite order matrices in $\mathrm{GL}(n, \Z)$. Define the virtually abelian groups $A=N \rtimes \langle x \rangle$ and $B=N \rtimes \langle y \rangle$ using the natural actions of $x$ and $y$ on $\Z^n$, and let $G=A*_{N} B$ denote the amalgamated product of $A$ and $B$ along  $N$. Then, by \Cref{cor:LR_for_amalgam}, $G$ has (LR) if and only if the subgroup $\langle x,y \rangle$ is finite in   $\mathrm{GL}(n, \Z) \cong \mathrm{Out}(N)$.
\end{ex}

\subsection{Property (VRC)}
The rest of the paper is devoted to investigating property (VRC). 
We start by characterizing it as a certain residual property:
\begin{cor}\label{cor: VRC "residually virtually abelian} The following are equivalent for a group $G$:	
		\begin{itemize}
			\item[(i)] $G$ has (VRC);
			\item[(ii)] for every $g \in G$ there is a finitely generated virtually abelian group $P$ and a homomorphism $\varphi:G \to P$ such that $\varphi$ is injective on $\langle g \rangle$;
			\item[(iii)] for every finitely generated virtually abelian subgroup $H \leqslant G$ there is a finitely generated virtually abelian group $P$ and a homomorphism $\varphi:G \to P$ such that $\varphi$ is injective on $H$.
		\end{itemize}
	\end{cor}
Thus a group $G$ has (VRC) if and only if for every element $g \in G$ there is a homomorphism to a finitely generated virtually abelian group $P$ preserving the order of $g$. This characterization turns out to be very useful and shows that (VRC) is a generalization of the \emph{RFRS} condition, which was employed in various virtual fibering results by Agol \cite{Agol-virt_fib}, Kielak \cite{Kielak}, Fisher \cite{Fisher}, and others. Indeed, in \cite[Theorem~6.3]{Ok-Sch} it is proved that a finitely generated group is RFRS  if and only if it is residually (locally indicable virtually abelian). Since locally indicable groups are torsion-free, we can conclude that finitely generated RFRS groups have (VRC). However, the class of (VRC) groups is much larger, as it includes many wreath products that need not be virtually torsion-free (see \cite[Theorem~9.4]{virtprops}). We also exhibit finitely generated locally indicable free-by-free groups with (VRC) that are not virtually RFRS (see \Cref{ex:VRC_not_virt_RFRS} below).

Property (VRC) is much more robust than (LR): it is stable under taking arbitrary subgroups, finite index supergroups, direct products and fundamental groups of finite graphs of groups with finite edge groups \cite{virtprops}. Minasyan \cite{virtprops} showed that every 
\emph{virtually special} group, in the sense of Haglund and Wise \cite{Haglund-Wise} (i.e., a group virtually embedded in a right angled Artin group), has (VRC), which includes many groups studied in Geometric Group Theory. 

However, (VRC) does not always behave nicely under extensions. For instance, a virtually polycyclic group has (VRC) if and only if it is virtually abelian \cite[Proposition~9.1]{virtprops}. 

\begin{ex}
The group $G$ from \Cref{ex:amalgam_of_two_virt_Z^n_gps} has (VRC) if and only if it has (LR). Indeed, if $G$ has (VRC) then the abelian normal subgroup $N  \n G$ must be a virtual retract, by \cite[Proposition~1.5]{virtprops}. It follows that the image of $G$ in $\mathrm{Out}(N)$ must be finite (see \Cref{lem:nvr->finite_image}), so $x$ and $y$ generate a finite subgroup of $\mathrm{GL}(n, \Z)$.
\end{ex}

In \Cref{sec:edge-approx} we develop a new criterion for establishing (VRC) in fundamental groups of finite graphs of groups. 

\begin{thm} \label{thm:VRC_for fund_gps_in_intro}
Let $G$ be the fundamental group of a finite graph of groups $(\mathcal{G},\Gamma)$. Suppose that there exist a group $A$ with (VRC) and a homomorphism $\varphi:G \to A$ such that 
\begin{itemize}
    \item $\varphi$ is injective on every vertex group $G_v$, $v \in V\Gamma$;
    \item the $\varphi$-image of each edge group $G_e$, $e \in E\Gamma$, is closed in the profinite topology on $A$.
\end{itemize}
Then $G$ has (VRC).
\end{thm}

If the vertex groups in a graph of groups $(\mathcal{G},\Gamma)$ are all finitely generated and virtually abelian, then we deduce that the fundamental group $G$ has (VRC) if and only if every vertex group $G_v$ is a virtual retract of $G$, $v \in V\Gamma$ (see \Cref{cor:homom_to_vab->VRC}). We can also use \Cref{thm:VRC_for fund_gps_in_intro} to produce finitely generated (VRC) groups with surprising properties, e.g., with unsolvable word problem: see \Cref{ex:VRC_with_unsolavble_WP}.

In \Cref{sec:trees_of_ab_gps} we prove that the fundamental group of a finite tree of finitely generated abelian groups always has (VRC) (\Cref{thm:tree of abelian->VRC}). The situation when the underlying graph of a graph of groups is not a tree is more complex. To study it, in \Cref{sec:action_on_R^n} we develop a tool for verifying property (VRC) using Euclidean geometry.

\begin{thm}\label{thm:hom_to_Rn_semidir_fin} Suppose that $G$ is the fundamental group of a finite graph of groups $(\mathcal{G},\Gamma)$, where all vertex groups are finitely generated  and virtually abelian. Then the following are equivalent:
\begin{itemize}
    \item[(i)]  $G$ has (VRC);
    \item[(ii)] there exist a Euclidean-by-finite group $P$ and a homomorphism \mbox{$\varphi:G \to P$} such that for all $v \in V\Gamma$ the restriction of $\varphi$ to the vertex group $G_v$ is injective and $\varphi(G_v)$ is a discrete subgroup of $P$;
    \item[(iii)] there exists a homomorphism $\varphi:G \to \R^n \rtimes Q$, for some finite group $Q$ acting on $\R^n$ linearly (i.e., via some, possibly non-faithful, representation $Q \to \mathrm{GL}(n,\R)$), such that $\varphi$ is injective on every vertex group $G_v$, $v \in V\Gamma$.
\end{itemize}
\end{thm}

Here a \emph{Euclidean-by-finite group} is a group of the form $\R^n \rtimes Q$, where $n \in \N_0$ and  $Q$ is a finite group acting on $\R^n$ linearly by Euclidean isometries (see \Cref{def:Eucl-by-fin}). 

\begin{notation}\label{not:C}
Let $\mathfrak{C}$ denote the class of groups decomposing as fundamental groups of finite graphs of groups where all vertex groups are finitely generated and virtually abelian.   
\end{notation}

This class $\mathfrak C$ is quite large and includes many groups of interest. For instance, it contains all tubular groups (see below), all (finitely generated abelian)-by-(finitely generated free) groups and many CAT($0$) groups. Every group from $\mathfrak C$ is finitely presented, by definition, and Bass-Serre theory implies that this class is closed under taking finitely generated subgroups. 

According to \Cref{thm:hom_to_Rn_semidir_fin}, if a group $G \in \mathfrak{C}$  has (VRC) then it admits an action on $\R^n$ by affine Euclidean isometries, such that the linear part of this action is finite and a finite index subgroup of each vertex group acts faithfully and discreetly by translations (see \Cref{rem:E-by-f_acts_by_isoms}).

In the remainder of \Cref{sec:action_on_R^n} we apply \Cref{thm:hom_to_Rn_semidir_fin} to decide (VRC)-ness in various examples. For instance, in \Cref{ex:G_kl} we show that, given any pair of integers $(k,l) \in \Z^2\setminus \{(0,0)\}$, the group
\begin{equation}\label{eq:G_kl}
G_{k,l} \coloneq \langle a,b,s,t \mid [a,b]=1,~sa s^{-1} = b,~t b t^{-1} =b^k a^l  \rangle. 
\end{equation}
has (VRC) if and only if $k,l \in \{0, \pm 1\}$. Note that $G_{k,l} \in \mathfrak{C}$ is a double HNN-extension of $\langle a,b \rangle \cong \Z^2$ with cyclic associated subgroups. Thus the group $G_{2,1}$ does not have (VRC), but since the elements $a$, $b$ and $b^2a$ are not proper powers in $\langle a,b \rangle$, we know that $G_{2,1}$ is cyclic subgroup separable by the work of Kim \cite[Theorem~3.7]{Kim-isolated}. This gives a negative answer to \cite[Question~11.6]{virtprops}. The latter question has been recently independently answered by Wu and Ye \cite{Wu-Ye}, using Gersten's example $H$ of a free-by-cyclic non-CAT($0$) group \cite{Gersten}, which can also be expressed as a double HNN-extension of $\Z^2$ with cyclic edge groups:
\begin{equation}\label{eq:Gersten}
H=\langle a,b,s,t \mid [a,b]=1,~sb s^{-1} = ab,~t b t^{-1} = a^2 b\rangle.    
\end{equation}
In fact, Gersten's original argument in \cite{Gersten} shows that $H$ cannot satisfy condition (ii) of  \Cref{thm:hom_to_Rn_semidir_fin}.

Any group $G$ with (VRC) is necessarily \emph{balanced}, that is, for every infinite order element $g \in G$ if $g^k$ is conjugate to $g^l$ in $G$ then $l=\pm k$ (see \Cref{rem:VRC->balanced}). In \Cref{sec:HNNs} we use \Cref{thm:hom_to_Rn_semidir_fin} to give an easy-to-check ``near linear independence'' criterion ensuring that a given multiple HNN-extension of a finitely generated abelian group has (VRC) (\Cref{thm:nearly_lin_indep_crit}). We apply this criterion to show that a single HNN-extension of a finitely generated abelian group has (VRC) if and only if it is balanced.

In \Cref{sec:gen_graphs_of_gps} we extend the results of \Cref{sec:HNNs} to general graphs of finitely generated abelian groups, using the observation that the fundamental group of such a graph maps onto a multiple HNN-extension of the abelianization of $(\mathcal{G}_T,T)$, where $T$ is any maximal tree in $\Gamma$, and this map is injective on the union of the vertex groups (\Cref{lem: graph of groups - tree - hom to v.a}). 

\begin{cor}\label{cor:balanced_circuits->VRC} Suppose that  $\Gamma$ is a connected graph with Euler characteristic $\chi(\Gamma) \ge 0$, and $(\mathcal{G},\Gamma)$ is a finite graphs of groups with finitely generated abelian vertex groups and with cyclic edge groups. Then the fundamental group $G$ of this graph of groups has property (VRC) if and only if $G$ is balanced.
\end{cor}  

The condition that $\chi(\Gamma) \ge 0$ means that $\Gamma$ can be obtained from a tree by adding at most one edge (and its inverse). The group $G_{2,1}$ from \eqref{eq:G_kl} shows that
this bound on $\chi(\Gamma)$ is necessary. The
assumption that all edge groups are cyclic in \Cref{cor:balanced_circuits->VRC} can be replaced by the weaker condition that
$G_e$ is cyclic for at least one edge outside a maximal tree in $\Gamma$ (see \Cref{cor:one_edge_outside_T}). It is also easy to decide if the group $G$ from \Cref{cor:balanced_circuits->VRC} is balanced: see \Cref{rem:determining_balanced}.

\subsection{Applications of property (VRC)}
The last two sections of the paper and the appendix are dedicated to giving applications of property (VRC), motivating our work.
\Cref{thm:hom_to_Rn_semidir_fin} already shows that for groups in the class $\mathfrak{C}$ the \`{a} priori algebraic property (VRC) has close connections with geometry. In \Cref{sec:CAT(0)} we further confirm this idea by proving

\begin{prop}\label{prop:VRC->CAT(0)} If a group $G \in \mathfrak{C}$ has (VRC) then it 
admits a proper and cocompact action by isometries on a complete CAT($0$) space.
\end{prop}

Thus, fundamental groups of finite graphs of finitely generated virtually abelian groups with (VRC) enjoy the many known properties of CAT($0$) groups, such as having at most quadratic Dehn functions, having solvable conjugacy problem, and so on (see \cite[Chapter~III.$\Gamma$.1]{Bridson_Haefliger}). Gersten's example \cite{Gersten} mentioned above shows that requiring (VRC) in \Cref{prop:VRC->CAT(0)} is essential. General groups from $\mathfrak{C}$ can have ``wild'' Dehn functions \cite{Bra-Bri} and unsolvable conjugacy problem \cite[Corollary~7.6]{BMV}.

\Cref{prop:VRC->CAT(0)} shows that every virtually special group from $\mathfrak C$ is CAT($0$). In the particular case of tubular groups this was first proved by Wise \cite[Lemma~4.4]{Wise-tubular}.
Recall that a \emph{tubular group} is the fundamental group of a non-empty finite graph of groups where all vertex groups are free abelian of rank $2$ and all edge groups are infinite cyclic. 
This class of groups is known to provide a rich source of examples. For instance, Gersten's group \eqref{eq:Gersten}, Wise's non-Hopfian CAT$(0)$ group \cite{Wise-non-Hopf}, Brady-Bridson groups, having a rich variety of Dehn functions \cite{Bra-Bri}, and the groups $G_{k,l}$, defined in \eqref{eq:G_kl}, are all tubular. 

In \Cref{prop:tubular+VRC->virt_free-by-cyclic}, using the work of Button \cite{Button-free-by-cyclic}, we show that any tubular group with (VRC) is virtually (free of finite rank)-by-cyclic (the adjective ``virtually'' is important: see \Cref{rem:virt_matters}). In \Cref{ex:which_G_kl_are_free-by-Z} we apply this to determine all $k,l$ for which the groups $G_{k,l}$ are free-by-cyclic.

In \cite{Wise-tubular} Wise characterized when a tubular group
admits a free action on a cube complex, and in  \cite{Woodhouse} Woodhouse gave criteria ensuring that a tubular group is virtually special. 
In  \Cref{sec:appendix} we show that if a group $H \in \mathfrak{C}$ has (VRC) then $H$ has a finite index subgroup acting freely on a finite dimensional CAT($0$) cube complex (\Cref{lem:VRC->free_action_on_cube_complex}). Combined with Woodhouse's work, this immediately implies the next statement.

\begin{prop}\label{prop:tubular_+VRC->virt_spec}
If $H$ is a tubular group with (VRC) then $H$ is virtually special. In other words, $H$ has a finite index subgroup which embeds in a finitely generated right angled Artin group.  \end{prop}
Thus, for tubular groups the conditions of having property (VRC) and being virtually special are equivalent. This gives a new way to establish virtual specialness of a tubular group either by using \Cref{thm:hom_to_Rn_semidir_fin} or, more directly, by verifying that a finite index subgroup of each vertex group survives in the abelianization of some finite index subgroup (see Remarks~\ref{rem:algorithm} and \ref{rem:Gardam}).

In \cite[Theorem~2]{Button-free-by-cyclic} Button proved that if the underlying graph of a tubular group is a tree then this group is virtually special. By combining \Cref{prop:tubular_+VRC->virt_spec} with \Cref{cor:balanced_circuits->VRC}, we are able extend this result to graphs of Euler characteristic $0$:

\begin{cor} Let $G$ be a tubular group whose underlying graph $\Gamma$ has $\chi(\Gamma) = 0$. Then $G$ is virtually  special if and only if  it is balanced.
\end{cor}

\ifnum \anonym=0
\subsection*{Acknowledgments} The authors are very grateful to Mark Hagen, Peter Kropholler, Ian Leary,  and Nansen Petrosyan, discussions with whom have greatly contributed to the development of this paper. We also thank Giles Gardam for creating a GAP implementation of the algorithm from \Cref{rem:algorithm}. Finally, we would like to thank the anonymous referee for their comments.
\fi
    
	\section{Background}
\subsection{Profinite topology}     Let $G$ be a group. For a subgroup $K \leqslant G$ we will write $K \leqslant_f G$ ($K \n_f G$) to indicate that $K$ has finite index  (is normal and has finite index) in $G$.     
The \emph{profinite topology} on $G$ is the topology generated by cosets of the finite index subgroups of $G$. Under this topology, translations by elements of the group,  inversion and group homomorphisms become continuous maps.

A subset $S \subseteq G$ is said to be \emph{separable} in $G$ if it is closed in the profinite topology on $G$. This means that for any $f\in G\setminus S$, there exists a finite group $P$ and a homomorphism $\psi\colon G \to P$ such that $\psi(f) \notin \psi(S)$ (equivalently, $f \notin S N$, where $N\coloneq \ker\psi \n_f G$).

     The following terminology is standard for a group $G$:
		\begin{itemize}
			\item $G$ is said to be \emph{residually finite} if $\{1\}$ is separable in $G$;
			\item $G$ is \emph{cyclic subgroup separable} if $\langle g \rangle$ is separable in $G$, for every $g\in G$;

			\item $G$ is \emph{LERF} (or \emph{subgroup separable}) if all finitely generated subgroups of $G$ are separable.
		\end{itemize}

\begin{rem}\label{rem:normal_sbgp_sep<->quot_rf}
A normal subgroup $B$ of a group $A$ is separable if and only if the quotient group $A/B$    is residually finite.
\end{rem}

The next result is an easy consequence of the definitions and we leave its proof as an exercise for the reader.
\begin{lemma}\label{lem:RF+sep}
Let $G$ be a group with a separable subset $S \subseteq G$ and let $F \subseteq G$ be any finite subset.
Then there exists a finite index normal subgroup $N \n_f G$ such that $F \cap SN=F \cap S$.
Moreover, if $G$ is residually finite then we can also ensure that the quotient map $\psi:G \to G/N$  is injective on $F$.   
\end{lemma}

		\subsection{Graphs of groups}\label{subsec:graphs_of_gps}
        Following Serre \cite[Section~I.2]{Serre},  we consider a graph $\Gamma$ as a pentuple $(V\Gamma,E\Gamma,\overline{\phantom{e}}, \alpha,\omega)$, where $V\Gamma$ is the \emph{set of vertices}, $E\Gamma$ is the \emph{set of edges}, $\overline{\phantom{e}}:E\Gamma \to E\Gamma$ is the \emph{inversion map} and $\alpha,\omega:E\Gamma \to V\Gamma$ are the \emph{incidence maps}, such that $\overline{e} \neq e$, $\overline{\overline{e}}=e$ and $\omega(\overline{e})=\alpha(e)$, for all $e \in E\Gamma$.
    Given an edge $e \in E\Gamma$, the vertices $\alpha(e)$ and $\omega(e)$ are called the \emph{initial} and \emph{terminal vertices} of $e$, respectively, and the edge $\ov{e}$ is called the \emph{inverse} of $e$.

A graph $\Gamma$ is a \emph{tree} if it is connected and contains no non-trivial closed paths without backtracking (a path has \emph{backtracking} if two consecutive edges form a mutually inverse pair), see \cite[Definition~6 in Section~I.2.2]{Serre}.

An \emph{orientation} on the edges of  a graph $\Gamma$ amounts to taking a representative from each pair of mutually inverse edges: $E\Gamma=E\Gamma^+ \sqcup E\Gamma^-$, where $E\Gamma^-=\{\ov{e} \mid e \in E\Gamma^+\}$.

Recall (\cite[Section~I.1]{BASS1993} or \cite[Section~2.16]{Bogopolski}) that a \emph{graph of groups} $ (\mathcal{G},\Gamma)$ consists of a non-empty connected graph $\Gamma$, along with a ``data set'' 
\begin{equation}\label{eq:data_for_graph_of_gps}
 \mathcal{G}= \Bigl(\{G_v\}_{v\in V\Gamma},\{G_e\}_{e \in E\Gamma} ,\{\alpha_e\}_{e \in E\Gamma}\Bigr),    
\end{equation}
where $\{G_v\}_{v\in V\Gamma}$ and $\{G_e\}_{e \in E\Gamma}$ are sets of groups, called  \emph{vertex groups} and \emph{edge groups} respectively, such that $G_e = G_{\overline{e}}$, for each $e \in E\Gamma$, and $ \alpha_e\colon~G_e~\to~G_{\alpha(e)}$, ${e\in E\Gamma}$, are monomorphisms. Similarly to the notation for graphs, we let $\omega_e:G_e \to G_{\omega(e)}$ denote the monomorphism $\alpha_{\overline{e}}$.

When the underlying graph $\Gamma$ is finite, we say that $(\mathcal{G},\Gamma)$ is a \emph{finite graph of groups}. We will use the version of the fundamental group of $(\mathcal{G},\Gamma)$ that is defined with the help of a \emph{maximal tree} $T$ in $\Gamma$ and an orientation on the edges $E\Gamma =E\Gamma^+ \sqcup E\Gamma^-$.
According to this definition, the fundamental group $\pi_1(\mathcal{G},\Gamma,T, E\Gamma^+)$ has the presentation
\begin{multline}\label{eq:pres_of_fund_gp}
\pi_1(\mathcal{G},\Gamma,T, E\Gamma^+)=\langle \{G_v\}_{v \in V\Gamma},~\{t_e\}_{e \in E\Gamma^+} \mid t_e \omega_e(a) t_e^{-1}=\alpha_e(a),\\  t_f=1,
\text{ for all }e \in E\Gamma^+,~a \in G_e~\text{ and } f \in ET \cap E\Gamma^+  \rangle.   
\end{multline} The elements $t_e$, $e \in E\Gamma^+$, are called the \emph{free letters} of $\pi_1(\mathcal{G},\Gamma,T, E\Gamma^+)$. 

The standard presentation of the fundamental group depending on the maximal tree $T$, $\pi_1(\mathcal{G},\Gamma,T)$, from \cite[Section~I.5.1]{Serre} or \cite[Section~2.16]{Bogopolski}, can be readily seen to be equivalent to ours, after applying a sequence of Tietze transformations to eliminate all free letters $t_e$, with $e \in E\Gamma^-$. It is also well-known that the fundamental group of a graph of groups does not depend on the choice of a maximal tree, up to isomorphism (see \cite[Proposition~20 in Section~I.5.1]{Serre}). Therefore, we may prove statements about the fundamental group of a graph of groups $(\mathcal{G},\Gamma)$ without specifying the choices for $T$ and an orientation on $E\Gamma$.

	\emph{Reduced expressions} are a powerful tool for working with fundamental groups of graphs of groups. For our purposes, it will be more convenient to use the version of the definition given in   \cite{Bogopolski}. The idea is to avoid subwords that can be ``easily'' shortened using the relations from the presentation \eqref{eq:pres_of_fund_gp} of $\pi_1(\mathcal{G},\Gamma,T, E\Gamma^+)$.

	Let $(\mathcal{G}, \Gamma)$ be a graph of groups as above with a fixed maximal subtree $T$ of $\Gamma$ and a fixed orientation $E\Gamma=E\Gamma^+ \sqcup E\Gamma^-$.     
    Let $ \sim $ denote the equivalence relation on $\bigsqcup_{v\in V\Gamma} G_v$ generated by
        \begin{equation*}
            \alpha_e(g) \sim \omega_e(g), ~\text{for every edge $e$ in $T$ and all $g \in G_e$}.
        \end{equation*}
    If $g \sim g' $ we say that these elements are \emph{equivalent with respect to $T$}.

\begin{rem}\label{rem:equivalence} 
    Presentation \eqref{eq:pres_of_fund_gp} shows that if two elements $g,g' \in \bigsqcup_{v\in V\Gamma} G_v$ are equivalent then their images in $G$ are the same. The converse of this is also true: see \Cref{lem: reduced expression} below.
\end{rem}

\begin{defn}[{\cite[Subsection~2.16.8]{Bogopolski}}]\label{reduced expressions}		
Let $x \in G = \pi_1(\mathcal{G},\Gamma,T, E\Gamma^+)$ be any element. Then $x$ can be written as a product  
\begin{equation}\label{eq:red_expr}
y_1 y_2 \dots y_n , ~\text{ where } y_i \in \bigsqcup_{v\in V\Gamma} (G_v\setminus\{1\}) \sqcup \{t_e^{\pm 1}\}_{e \in E\Gamma^+\setminus ET},
\end{equation}
for each $i=1,\dots, n$. The integer  $n \in \N_0$ is called the \emph{length} of the expression \eqref{eq:red_expr}, and the trivial element of $G$ is represented by the empty expression of length $0$.

The expression \eqref{eq:red_expr} is said to be \emph{reduced} in  $(\mathcal{G},\Gamma)$ (with respect to $T$ and the chosen orientation of $E\Gamma$) if it satisfies the following:
		\begin{itemize}
			\item if for some $i \in \{1,\dots,n-1\}$, $y_i,y_{i+1} \in \bigsqcup_{v\in V\Gamma} G_v$ then these elements are not equivalent to elements of the same vertex group, with respect to the equivalence relation defined above;
			
            \item the expression \eqref{eq:red_expr} does not contain subwords of the form $t_e t_e^{-1} $ or $ t_e^{-1} t_e $, for any $e \in E\Gamma^+$;
			
			\item for any $e \in E\Gamma^+$, the expression \eqref{eq:red_expr} does not contain subwords of the form $ t_e^{-1} g t_e $, where $g\in \bigsqcup_{v\in V\Gamma} G_v$ is equivalent to an element of $ \alpha_e(G_{e})$;
			
			\item for any $e \in E\Gamma^+$, the expression \eqref{eq:red_expr} does not contain subwords of the form $ t_e g t_e^{-1} $, where $g\in \bigsqcup_{v\in V\Gamma} G_v$ is equivalent to an element in $ \omega_e(G_{e})$.
		\end{itemize}

    \end{defn}
Note that if an expression is not reduced then it can be shortened. This shows that every element of $G$ has a reduced expression.

\begin{lemma}[{\cite[Theorem~16.10]{Bogopolski}}]
\label{lem: reduced expression}
        If $x \in G=\pi_1(\mathcal{G},\Gamma,T, E\Gamma^+)$ has a non-empty reduced expression, then $x\ne 1$ in $G$. It follows that the vertex groups $G_v$ are naturally embedded in $\pi_1(\mathcal{G},\Gamma,T, E\Gamma^+)$, and two elements $g,g' \in \bigsqcup_{v\in V\Gamma} G_v$ are equivalent (with respect to the equivalence relation $\sim$ defined above) if and only if they represent the same element of $G$.
\end{lemma}

In view of \Cref{lem: reduced expression}, we can (and often will) treat the vertex groups $G_v$ as subgroups of the fundamental group $G=\pi_1(\mathcal{G},\Gamma,T, E\Gamma^+)$.

\begin{defn}\label{def:cyc_reduced}		    
We will say that the expression \eqref{eq:red_expr}  is \emph{cyclically reduced} in $(\mathcal{G},\Gamma)$ (with respect to $T$ and the orientation of $E\Gamma$)  if every cyclic permutation of $y_1y_2 \dots y_n$ is reduced in the sense of \Cref{reduced expressions}.
\end{defn}

Evidently, every element of $G=\pi_1(\mathcal{G},\Gamma,T, E\Gamma^+)$ is conjugate to an element that has a cyclically reduced expression. If $n \ge 2$ and an expression $y_1y_2 \dots y_n$ is cyclically reduced, then for each $m \in \N$ the expression $(y_1y_2 \dots y_n)^m$ is reduced. 

\begin{defn}\label{def:ell-hyp}		    
An element $x \in G=\pi_1(\mathcal{G},\Gamma,T, E\Gamma^+)$ is called \emph{elliptic} if it is conjugate in $G$ to an element from a vertex group, otherwise $x$ is called \emph{hyperbolic}.
\end{defn}

\begin{rem}\label{rem:cycl_reduced->powers_are_reduced} Definitions~\ref{def:cyc_reduced} and \ref{def:ell-hyp} and \Cref{lem: reduced expression} imply the following useful facts for $G \coloneq \pi_1(\mathcal{G},\Gamma,T, E\Gamma^+)$:
\begin{itemize}
    \item for any given element $x \in G$, every reduced expression representing $x$  has the same length;
    \item if an element $x \in G$ is elliptic then the length of $x^m$ is uniformly bounded, for all $m \in \N$;
    \item a cyclically reduced expression \eqref{eq:red_expr} represents a hyperbolic element of $G$ if and only if either $n \ge 2$  or $n=1$ and  $y_1=t_e^{\pm 1}$, for some $e \in E\Gamma^+\setminus ET$;
    \item every hyperbolic element of $G$ has infinite order.
\end{itemize}    
\end{rem}

		Here is a more advanced result relevant to this paper.
	\begin{thm}[{\cite[Theorem~1]{Solitar73}}]\label{thm:virtually free}
		Let $G$ be a finitely generated group. Then $G$ is virtually free if and only if $G$ is isomorphic to the fundamental group of a finite graph of groups $(\mathcal{G},\Gamma)$, where all vertex groups $G_v$ are finite.
	\end{thm}

\subsection{Morphisms between graphs of groups}
In \cite[Section~2]{BASS1993} Bass  defined a notion of morphisms between graphs of groups that induces homomorphisms between their fundamental groups (see  \cite[Proposition~2.4]{BASS1993}). We will use such morphisms in the two special cases below. In both of these cases the underlying connected graph $\Gamma$ is unchanged, so fixing the same maximal tree $T$,
the existence of the homomorphism between the fundamental groups can be easily extracted from the presentation \eqref{eq:pres_of_fund_gp}.

\begin{ex}[Quotient graph of groups]
\label{ex:morphism_between_fund_gps-1} Suppose that $(\mathcal{G},\Gamma)$ is a graph of groups, defined by \eqref{eq:data_for_graph_of_gps}, and for each vertex $v \in V\Gamma$ we have a normal subgroup $M_v \lhd G_v$ such that $M_v \cap \alpha_e(G_e)=\{1\}$, for all $e \in E\Gamma$ with $\alpha(e)=v$ (that is, $M_v$ has trivial intersection with the images of all edge groups incident to $v$ in $\Gamma$). 

For every $v \in V\Gamma$, set $\ot{G}_v=G_v/M_v$ and fix an epimorphism $\varphi_v:G_v \to \ot{G}_v$, with kernel $M_v$. Given any $e \in E\Gamma$, the map $\ot{\alpha}_e\coloneq \varphi_{\alpha(e)} \circ \alpha_e:G_e \to \ot{G}_v$ is a monomorphism because $\varphi_{\alpha(e)}$ is injective on $\alpha_e(G_e)$. We now define the \emph{quotient graph of groups} $(\ot{\mathcal{G}},\Gamma)$, corresponding to the family of normal subgroups $\{M_v\}_{v \in V\Gamma}$ in the vertex groups, with underlying graph $\Gamma$ and data set
\begin{equation*}
 \ot{\mathcal{G}}= \Bigl(\{\ot{G}_v\}_{v\in V\Gamma},\{G_e\}_{e \in E\Gamma} ,\{\ot{\alpha}_e\}_{e \in E\Gamma}\Bigr).  
\end{equation*}

After fixing a maximal tree $T$ in $\Gamma$ and an orientation $E\Gamma=E\Gamma^+\sqcup E\Gamma^-$, we obtain a group homomorphism
\[\varphi: \pi_1(\mathcal{G},\Gamma,T, E\Gamma^+) \to \pi_1(\ot{\mathcal{G}},\Gamma,T,E\Gamma^+),\] such that the free letters are sent to themselves and the restriction of $\varphi$ to $G_v$ is $\varphi_v$, for all $v \in V\Gamma$.
\end{ex}

\begin{ex}[Induced graph of groups]
\label{ex:morphism of graphs}
Let $G=\pi_1(\mathcal{G},\Gamma, T, E\Gamma^+)$ be the fundamental group of a graph of groups $(\mathcal{G},\Gamma)$,  defined by \eqref{eq:data_for_graph_of_gps}, with respect to a maximal tree $T$ in $\Gamma$ and an orientation $E\Gamma=E\Gamma^+ \sqcup E\Gamma^-$.
        
Suppose that we a have a group homomorphism $\rho:G \to K$, for some group $K$. Then we can construct a new graph of groups $(\ot{\mathcal{G}},\Gamma)$, \emph{induced by the homomorphism $\rho:G \to K$}, with the same underlying graph $\Gamma$, as follows. For every $v \in V\Gamma$, we set $\ot{G}_v \coloneq \rho(G_v) \leqslant K$, and for each edge $e \in E\Gamma^+$, we let $\ot{G}_e=\ot{G}_{\ov{e}} \coloneq \rho(\alpha_e(G_e)) \leqslant K$. Now, for every $e \in E\Gamma^+$, we denote by $\ot{\alpha}_e:\ot{G}_e \to \ot{G}_{\alpha(e)}$ the inclusion map and we define the monomorphisms $\ot{\omega}_e:\ot{G}_e \to \ot{G}_{\omega(e)}$ by setting
\begin{equation}\label{eq:def_of_tilde_omega}
\ot{\omega}_{e}(g) \coloneq \rho(t_e)^{-1}\, g \, \rho(t_e),~\text{ for all } g \in \ot{G}_e   \end{equation}
(note that $\ot{\omega}_e(g) \in \ot{G}_{\omega(e)}$ in $K$ because $G$ has presentation \eqref{eq:pres_of_fund_gp}). 

We can now set $\ot{\alpha}_e \coloneq \ot{\omega}_{\ov{e}}$, for each $e \in E\Gamma^-$ and define the graph of groups $(\ot{\mathcal{G}},\Gamma)$, where the vertex groups are $\{\ot{G}_v\}_{v \in V\Gamma}$, the edge groups are $\{\ot{G}_e\}_{e \in E\Gamma}$ and the monomorphisms are $\{\ot{\alpha}_e\}_{e \in E\Gamma}$. 
Again, the restriction of $\rho$ to the vertex (and edge) groups can be used to define a homomorphism between the fundamental groups of $({\mathcal{G}},\Gamma)$ and $(\ot{\mathcal{G}},\Gamma)$. More precisely, 
using the same maximal tree $T$ in $\Gamma$ and the same orientation of edges, we have 
\begin{multline*}\label{eq:pres_of_fund_gp}
\pi_1(\ot{\mathcal{G}},\Gamma,T, E\Gamma^+)=\langle \{\ot{G}_v\}_{v \in V\Gamma},~\{\ot{t}_e\}_{e \in E\Gamma^+} \mid \ot{t}_e \ot{\omega}_e(a)\ot{t}_e^{-1} =\ot{\alpha}_e(a),\\  \ot{t}_f=1,
\text{ for all }e \in E\Gamma^+,~a \in \ot{G}_e~\text{ and } f \in ET \cap E\Gamma^+  \rangle.   
\end{multline*}
As evident from the presentations, we have a group homomorphism $\varphi:G \to \ot{G} \coloneq \pi_1(\ot{\mathcal{G}},\Gamma,T, E\Gamma^+)$, defined by assigning
\begin{equation*}\label{eq:def_of_map_varhi]}
\varphi(t_e)=\ot{t}_e, \text{ for all } e \in E\Gamma^+, \text{ and }
\varphi|_{G_v}=\rho|_{G_v}, \text{ for all } v \in V\Gamma.     
\end{equation*}
Moreover, it is clear from the construction that the original homomorphism $\rho:G \to K$ factors through the homomorphism $\varphi: G \to \ot{G}$.
\end{ex}


	\section{Virtual retraction properties in groups}
	\begin{defn}
		Let $G$ be a group and let $H$ be a subgroup. We say that $H$ is a \emph{retract} of $G$ if there exists a homomorphism $\rho\colon G \to H$ that restricts to the identity on $H$. 
        
        The subgroup $H$ is a \emph{virtual retract} of $G$ if there exists a finite index subgroup $K\leqslant_f G$ such that $H \subseteq K$ and $H$ is a retract of $K$. In this case we will write $H \vr G$. If, additionally, $H$ is normal in $G$ then we will write $H \nvr G$.
	\end{defn}

		Observe that $H$ is a retract of $K$ if and only if there exists a normal subgroup $N\n K$ (the kernel of a retraction $K \to H$) such that $K = HN$ and $N\cap H = \{1\}$. The latter means that  $K$ splits as a semidirect product $N \rtimes H$.

\begin{defn}\label{def:virt_compl}
Given a group $G$ and a subgroup $H \leqslant G$, a \emph{virtual complement to $H$ in $G$} is any subgroup $N \leqslant G$ such that    $N$ is normalized by $H$, $H \cap N =\{1\}$ and $HN \leqslant_{f} G$.  
\end{defn}
Thus $H \vr G$ if and only if $H$ has a virtual complement in $G$.

	\begin{defn}
	    Let $G$ be a group. If all finitely generated subgroups of $G$ are virtual retracts of $G$, we say that $G$ has property \emph{(LR)}. If all cyclic subgroups are virtual retracts, we say that $G$ has property \emph{(VRC)}.
	\end{defn}

Property (LR) clearly implies (VRC), but not vice-versa.    
In the next two lemmas we collect some statements about groups with (VRC) and (LR), most of which were observed in \cite{virtprops}.

 \begin{lemma}\label{lem: props sep}
            The following statements are true for a group $G$.
            \begin{itemize}
                \item[(i)] If $H\vr G$ and $G$ is residually finite then $H$ is separable in $G$; see \cite[Lemma~2.2]{virtprops}.
        \item[(ii)] If $G$ has (VRC) then it is residually finite and cyclic subgroup separable. Moreover, for every  finitely generated virtually abelian subgroup $H$ in $G$ we have $H \vr G$; see \cite[Lemma~5.1 and Proposition~1.5]{virtprops}.
        \item[(iii)] If $G$ has  (LR) then $G$ is LERF; see  \cite[Lemma~5.1.(iii)]{virtprops}
            \end{itemize}
        \end{lemma}

	\begin{lemma}\label{lem: listresultsretr} Let $G$ be a group.
			\begin{itemize}
			\item[(i)] Suppose that $H\leqslant G$ and $\varphi\colon G \to G'$ is a homomorphism injective on $H$ such that $\varphi(H) \vr  G'$. Then $H\vr G$. In particular, if $\varphi$ is a monomorphism and $G'$ has (LR) [(VRC)] then  $G$ also has (LR) [respectively, (VRC)]; see \cite[Lemma~3.2.(ii)]{virtprops}. 
			
			\item[(ii)] If $F \vr H \vr G$ then $F \vr G$. In particular, if $H \vr G$ and $F \leqslant_f H$ is a subgroup of finite index then $F \vr G$; see \cite[Lemma~3.2.(iv)]{virtprops}.
            
			\item[(iii)] Finitely generated free groups have (LR); see M. Hall's theorem as stated in \cite[Theorem~1]{Burns}.
			
			\item[(iv)] Finitely generated virtually abelian groups have (LR); see \cite[Corollary~4.3]{virtprops}.
			\item[(v)] Finite free products of groups with (LR) have (LR); see \cite[Theorem~1.5]{GITIK2003}.
			\item[(vi)] The direct product of a group with (LR) and a finitely generated virtually abelian group has (LR); see \cite[Proposition~5.6]{virtprops}.
                \item[(vii)] Finitely generated virtually free groups have (VRC); see \cite[Corollary~5.3]{virtprops}.
\item[(viii)] Property (VRC) is preserved under direct products and under passing to arbitrary subgroups and to finite index supergroups; see \cite[Lemmas~5.1 and 5.2]{virtprops}.
		\end{itemize}
	\end{lemma}

Finally, property (VRC) is stable under taking amalgamated products and HNN-extensions with finite associated subgroups.

\begin{lemma}[{\cite[Corollary~6.5]{virtprops}}]\label{lem:VRC_preserved_by_free_contr_with_fin_edge_gps} Let $G$ be the fundamental group of a finite graphs of groups whose vertex groups have (VRC) and whose edge groups are finite. Then $G$ has (VRC).    
\end{lemma}

\begin{lemma}\label{lem:virt_retract->homom}
Suppose that $G$ is a group and $H \vr G$. Then for every virtual complement $N$ to $H$ in $G$, and for each finite index normal subgroup $J \n_f G$  there exist a group $P$ and an epimorphism $\varphi:G \to P$ such that 
\begin{itemize}
    \item a finite index subgroup of $P$ embeds in  the direct power $H^n$, for some $n \in \N$;
    \item $\ker\varphi \subseteq N \cap J$;    
    \item the restriction of $\varphi$ to $H$ is injective and $\varphi(H) \vr P$;
\end{itemize}
Moreover, if $H$ has (VRC) then so does $P$.
\end{lemma}

\begin{proof} The argument is similar to the proof of \cite[Lemma~4.1]{virtprops}.

Let $K\coloneq HN$ be a finite index subgroup of $G$ retracting onto $H$,    let $L \n_f G$ be the normal core of $K$ in $G$, and denote $N_1 \coloneq L\cap N\n L$. 
Let $g_1=1, g_2,\dots, g_n \in G$ be a list of coset representatives for $L$ in $G$, and set \[N_2 \coloneq \bigcap_{i=1}^n g_i N_1 g_i^{-1} \cap J.\] Since $L\n G$ and $N_1 \n L$, we have that $g_i N_1 g_{i}^{-1}$ is a normal subgroup of $L$, for each $i=1,\dots,n$. It follows that $N_2\n G$ and we can define $P \coloneq G/N_2$ and let $\varphi: G \to P$ be the quotient epimorphism. By construction, $\ker\varphi=N_2 \subseteq N \cap J$.    
    Since $N_2 \subseteq N_1 \subseteq N$, we see that $\varphi$ is injective on $H$ and $\varphi(K)$ retracts onto $\varphi(H)$ with kernel $\varphi(N)$, thus $\varphi(H) \vr P$. 
    
    To show that $P$ is commensurable with a subgroup of $H^n$, observe that $L/N_2$ has finite index in $P$, and we have a natural (diagonal) embedding 
    \begin{equation}\label{eq:emb_into_dir_prod}
        {L}/{N_2} \lhook\joinrel\longrightarrow \bigtimes_{i=1}^n {L}/(g_iN_1g_i^{-1}) \times L/(L \cap J).   
    \end{equation}
    Note that $L/g_iN_1g_i^{-1}= g_iLg_i^{-1}/g_iN_1g_i^{-1}\cong L/N_1$, for all $i$, so we have  
    \[O \coloneq \bigtimes_{i=1}^n {L}/(g_iN_1g_i^{-1}) \cong (L/N_1)^n.\]
    Since $|G:J|<\infty$, the group $L/(L \cap J)$ is finite, and \eqref{eq:emb_into_dir_prod} shows that a finite index subgroup of $L/N_2$ is isomorphic to a subgroup of $O$. By construction, $L/N_1\cong LN/N \n_f K/N \cong H$, so $O$ has finite index in $H^n$, as required.

    The final assertion holds because property (VRC) is preserved by taking direct products, subgroups and finite index supergroups (\Cref{lem: listresultsretr}.(viii)), so if $H$ has (VRC) then the same holds for $P$. 
\end{proof}

In the special case when $H$ is finitely generated and virtually abelian, \Cref{lem:virt_retract->homom} gives the following.    
	\begin{cor}
    \label{cor:retr_onto_virt_ab->homom}   
		Let $G$ be a group with $J \n_f G$, and let $H \vr G$ be a finitely generated virtually abelian subgroup.
		Then there is a finitely generated virtually abelian group $P$ and an epimorphism $\varphi:G \to P$ such that $\varphi$ is injective on $H$ and $\ker\varphi \subseteq J$.
	\end{cor}

	\Cref{cor:retr_onto_virt_ab->homom} allows us to characterize property (VRC) as an  approximation property by finitely generated virtually abelian groups, as mentioned in the Introduction.

	\begin{proof}[Proof of \Cref{cor: VRC "residually virtually abelian}]
            The sequence of this proof is (ii)$\Rightarrow$(i)$\Rightarrow$(iii)$\Rightarrow$(ii). The last implication is obvious.
                Condition (ii) implies (i), by parts (i) and (iv) of \Cref{lem: listresultsretr}. 
            The fact that (i) implies (iii) can be obtained by combining \Cref{lem: props sep}.(ii) with \Cref{cor:retr_onto_virt_ab->homom}. 
	\end{proof}

\Cref{cor:retr_onto_virt_ab->homom} can be improved. We state the proposition below in a slightly more general form than we require, with a view to future applications (for the purposes of this paper one can assume that $J=G$).
    
\begin{prop}\label{prop:further_props_of_VRC} Suppose that $G$ is a group with  (VRC). If $J \n_f G$, and $H_1,\dots,H_k \leqslant G$ is a finite collection of finitely generated virtually abelian subgroups, then there exist a finitely generated virtually abelian group $P$ and a homomorphism $\varphi:G \to P$ such that $\varphi$ is injective on $H_i$, for each $i=1,\dots,k$, and $\ker\varphi \subseteq J$.    
\end{prop}
    
\begin{proof} 
By \Cref{cor:retr_onto_virt_ab->homom}, for each $i=1,\dots,k$ there exist a finitely generated virtually abelian group $P_i$ and a homomorphism $\varphi_i:G \to P_i$ such that $\varphi_i$ is injective on $H_i$ and $\ker\varphi_i \subseteq J$. 
We can now define the homomorphism  
            \begin{equation*}
                \varphi \coloneq \varphi_1 \times \cdots \times \varphi_k:  G \to P\coloneq P_1 \times \cdots \times P_k .
            \end{equation*}
Evidently $P$ is a finitely generated virtually abelian group. And since the homomorphism $\varphi_i$ factors through $\varphi$, for each $i=1,\dots,k$, it is easy to check that $\varphi$ inherits all the required properties.
\end{proof}
	
\begin{rem}\label{rem:just_vr} As the proof of \Cref{prop:further_props_of_VRC} shows, we can substitute the requirement that $G$ has (VRC) by the weaker assumption that $H_i \vr G$, for each $i=1,\dots,k$.    
\end{rem}

	Now we state the following extension results, found in \cite[Lemma~5.8]{virtprops}.
	\begin{lemma}\label{extensions}
		Let $G$ be a group and let $N$ be a normal subgroup of $G$. Then the following hold.
		\begin{itemize}
			\item[(i)] If $G$ has (LR) and $N$ is finitely generated then $G/N$ has (LR).
			\item[(ii)] If $N$ is finite, $G$ is residually finite  and $G/N$ has (LR) then $G$ also has (LR). 
		\end{itemize}
	\end{lemma}
	In \Cref{cor:LR extension from quot} below  we will  improve the second statement of this lemma: the subgroup $N$ can be chosen to be virtually abelian as long as it is a virtual retract of $G$. To establish this, we will explore normal virtual complements in \Cref{sec:nvc}.


\section{Normal virtual complements}\label{sec:nvc}
Every virtual retract $H$ of a group $G$ has a virtual complement (see \Cref{def:virt_compl}). Sometimes this virtual complement can be chosen so that it is normalized not only by $H$ but by the whole of $G$.

\begin{defn}
		Let $G$ be a group and let $H\leqslant G$ be a subgroup. We say that $L \n G$ is a \emph{normal virtual complement} to $H$ in $G$ if  $H \cap L = \{1\}$ and $HL \leqslant_{f} G$. 
\end{defn}

The next lemma tells us that normal subgroups of finitely generated virtually abelian groups have normal virtual complements.
	\begin{lemma}[{\cite[Lemma~4.2]{virtprops}}] \label{lem:virt_ab_normal_retr}
		Let $P$ be a finitely generated virtually abelian group and let $S \n P$ be a normal subgroup. Then
		there is a torsion-free normal subgroup $R \n P$ such that $|P:SR|<\infty$ and $S \cap R$ is trivial.
	\end{lemma}

	\begin{lemma}\label{lem:virtual normal complement}
		Let $G$ be a group and let $N\nvr G$ be a finitely generated normal virtual retract. If $N$ is virtually abelian then it has a normal virtual complement in $G$.
	\end{lemma}
	\begin{proof}
		Using \Cref{cor:retr_onto_virt_ab->homom}, we get an epimorphism $\psi\colon G \to P$, where $P$ is a finitely generated virtually abelian group and $\psi$ is injective on $N$. Since $\psi(N)$ is normal in the virtually abelian group $P$, we can use \Cref{lem:virt_ab_normal_retr} to find a normal virtual complement $M\n P$ to $\psi(N)$ in $P$. Then it is easy to see that $L\coloneq \psi^{-1}(M)$ is a normal virtual complement to $N$ in $G$.
	\end{proof}

	We can now prove the desired generalization of \Cref{extensions}.(ii).
	
	\begin{cor}\label{cor:LR extension from quot}
		Let $G$ be a group and let $N\nvr G$ be a normal virtual retract of $G$. If $N$ is finitely generated and virtually abelian then there is a monomorphism $\iota \colon G \to G/N \times A$, where $A$ is a finitely generated virtually abelian group. In particular, if $G/N$ has (LR) then $G$ also has (LR).
	\end{cor}
	\begin{proof}
		By \Cref{lem:virtual normal complement}, we can find a normal virtual complement $L$ to $N$ in $G$. Then define $A\coloneq G/L$ and a homomorphism $\iota$ as
		\[
		\begin{tikzcd}[row sep=2pt]
			\iota\colon G \arrow[r,rightarrow] &G/N \times A, \\
			g \arrow[r,mapsto,"\iota"]  &(gN,gL) .
		\end{tikzcd} 
		\]
		The homomorphism $\iota$ is injective because $N \cap L=\{1\}$. 		Since $NL \leqslant_{f} G$, we know that $N\cong NL/L$ has finite index in $G/L=A$. Therefore, $A$ is virtually abelian and finitely generated (as this is true for $N$).
		
		For the last claim of the corollary, we deduce that $G/N \times A$ has (LR) from \Cref{lem: listresultsretr}.(vi), hence $G$ has (LR) by \Cref{lem: listresultsretr}.(i).
	\end{proof}
        We proceed with a few more observations concerning normal virtual retracts. Recall that for a group $N$, the \emph{outer automorphism group} $\Out(N)$ is defined as the quotient of the \emph{automorphism group} $\Aut(N)$, of $N$, by the subgroup of \emph{inner automorphisms} $\mathrm{Inn}(N)$. If $N$ is a normal subgroup in a group $G$ then the action of $G$ on $N$ by conjugation gives rise to a homomorphism $G \to \Aut(N)$. The composition of this homomorphism with the quotient map $\Aut(N) \to \Out(N)$ will be called the \emph{natural homomorphism} $G \to \Out(N)$. Note that kernel of this homomorphism is the subgroup $N\C_G(N) \n G$.

        \begin{lemma}\label{lem:nvr->finite_image}
            Consider a group $G$ and a normal subgroup $N\n G$. Let $\psi\colon G \to \Out(N)$ be the natural homomorphism induced by the action  of $G$ on $N$ by conjugation. If $N$ is a virtual retract of $G$ then $\psi(G)$ is finite.
        \end{lemma}
        \begin{proof}
            Since $N\nvr G$, there exist subgroups $M$ and $K$ in $G$ such that $M \n K\leqslant_f G$, $M \cap N=\{1\}$ and $K=MN$. Since $N$ is also normal in $K$, we have that $N$ and $M$ commute, whence $\psi(K)=\{1\}$. Recalling that $K\leqslant_f G$ and $\psi(N)=\{1\}$, we deduce that $\psi(G)$ must be finite.
        \end{proof}
        \begin{rem} \label{rem:finite_image->Ncent(N)-finite_index}
            Let $G$ be a group with a normal subgroup $N \n G$. Then $\C_G(N) \n G$ and       the natural homomorphism $G \to \Out(N)$ has finite image  
 if and only if  $|G:N\C_G(N)|<\infty$.          \end{rem}

\begin{cor} Let $N$ be a normal subgroup of a group $G$. Suppose that either $\cent(N) = \{1\}$ or $\cent(N)$ is finite and $G$ is residually finite. Then the following are equivalent:
\begin{itemize}
    \item[(i)] $N\vr  G$;
    \item[(ii)] the natural homomorphism $G \to \Out(N)$ has finite image;
    \item[(iii)] $N$ has a normal virtual complement in $G$.
\end{itemize}
\end{cor}

\begin{proof} Clearly, (iii) implies (i) and the implication (i) $\Rightarrow$ (ii) is given by \Cref{lem:nvr->finite_image}. Therefore, we only need to show that (ii) implies (iii). 

Suppose that the natural homomorphism $G \to \Out(N) $ has finite image. Then $\C_G(N)\n G$ and $|G:N\C_G(N)|<\infty$, by \Cref{rem:finite_image->Ncent(N)-finite_index}. Note that $N \cap \C_G(N)=\cent(N)$, so if $\cent(N)=\{1\}$ then $\C_G(N)$ will be a normal virtual complement to $N$ in $G$. If $\cent(N)$ is finite and $G$ is residually finite, then there exists a finite index normal subgroup $K \n G$ such that $K \cap \cent(N)=\{1\}$. We can then set $M\coloneq \C_G(N) \cap K \leqslant_f \C_G(N)$, and observe that $M \n G$ and $M \cap N=\{1\}$. Then $|G:NM|< |G:N\C_G(N)|\, |N\C_G(N): NM| < \infty$ and thus, $M$ is a normal virtual complement to $N$ in $G$.    
\end{proof}

More generally, we can obtain a converse of \Cref{lem:nvr->finite_image} and  strengthen the statement of \Cref{lem:virtual normal complement} as follows.
        \begin{lemma} \label{lem:norm_virt_complem-improved}
            Consider a group $G$ and a normal subgroup $N\n G$. Let $\psi\colon G \to \Out(N)$ be the natural homomorphism induced by the action  of $G$ on $N$ by conjugation. Suppose that $\psi(G)$ is finite, the centre $\cent(N)$ is finitely generated and $\cent(N) \vr G$. Then $N$ has a normal virtual complement in $G$.
        \end{lemma}
        \begin{proof}
            Since $N\n G$, we have that $\cent(N)\n G$, so, by \Cref{lem:virtual normal complement}, there exists a normal virtual complement $M$ to $\cent(N)$ in $G$. We claim that $L\coloneq M \cap \C_G(N)$ is a normal virtual complement to $N$. Indeed, since $N \n G$, we have $\C_G(N) \n G$, hence $L \n G$.
           By construction, we also have $L\cap N =M \cap \cent(N)=\{1\}$. Then, since $L \subseteq \C_G(N)$ and $|G:\cent(N)M|<\infty$, we obtain
            \begin{equation*}
                NL \cap \C_G(N) = (N \cap \C_G(N))L = \cent(N) L = \cent(N) M \cap \C_G(N) \leqslant_{f} \C_G(N).
            \end{equation*}
           It follows that $NL\leqslant_f N \C_G(N)\leqslant_f G$, by \Cref{rem:finite_image->Ncent(N)-finite_index}, and thus $NL\leqslant_f G$, as we wanted.
        \end{proof}

        \Cref{lem:norm_virt_complem-improved} shows that in many cases a normal virtual retract has a normal virtual complement, however we expect  that the answer to the following question is negative\footnote{Nicholas Touikan's work \cite{Touikan}, which appeared after this paper was completed, gives a surprising positive answer to \Cref{que:normal_virt_compl}}.

        \begin{question} \label{que:normal_virt_compl}
            Suppose that  $G$ is a finitely generated group and  $N\nvr G$. Must $N$ have a normal virtual complement in $G$?
        \end{question}
	
Having a normal virtual complement is generally much stronger than just being a virtual retract. This is demonstrated, for example, by the following statement.

\begin{prop}\label{prop:nvc->fi_supergp} Let $H$ be a subgroup of a group $G$ and let $L\lhd G$ be a normal virtual complement to $H$ in $G$. Suppose that at least one of the following conditions is satisfied:
\begin{itemize}
    \item[(a)] $L$ is torsion-free;
    \item[(b)] $G$ is residually finite;
    \item[(c)] $H$ is separable in $G$.
\end{itemize}
Then for arbitrary subgroups $J_1,\dots,J_k \leqslant G$ such that $H \leqslant_f J_i$, for every $i$, there is $M \n G$ which is a  normal virtual complement to each $J_i$, $i=1,\dots,k$, in $G$. In particular, $J_1,\dots,J_k \vr G$.
\end{prop}

\begin{proof} Since $L \cap H=\{1\}$ and $H$ has finite index in $J_i$, the intersection $F_i\coloneq L \cap J_i$ must be a finite subgroup of $L$, for each $i=1,\dots,k$.  
Denote $F \coloneq \bigcup_{i=1}^k F_i \subseteq G$, then $|F|<\infty$ and $F \cap H=\{1\}$.

If condition (a) is satisfied then $F=\{1\}$ and since $L \n G$ we see that $|G:J_i L|\le |G:HL|<\infty$, so
$L$ is a normal virtual complement to $J_i$ in $G$, for all $i=1,\dots,k$.
    
If one of (b) or (c) holds, then there is a finite index normal subgroup $K \n_f G$ with $K \cap F=\{1\}$, and we can set $M \coloneq K \cap L \leqslant_f L$. Then $M \n G$ and $|HL:HM|\le |L:M|<\infty$, so 
\[|G:J_iM|\le |G:HM| \le |G:HL|\,|HL:HM|<\infty,~i=1,\dots,k.\]
Thus $M$ is a normal virtual complement to each $J_i$ in $G$, as required.
\end{proof}

A construction from \cite[Example~3.6]{virtprops} shows that having a normal virtual complement is necessary in \Cref{prop:nvc->fi_supergp}.
    

\section{Property (LR) for some graphs of groups}

In this section we will look at fundamental groups $G$ of graphs of groups where the edge groups are all equal to some normal subgroup $N \n G$. We will use the concept of normal virtual complements developed in \Cref{sec:nvc} to give sufficient criteria ensuring that $G$ has property (LR).

\begin{lemma}\label{lem:norm_sbgps_with_virt_free_quot->vr} Suppose that $G$ is a group with a normal subgroup $N \n G$ such that $G/N$ is virtually free. If the natural homomorphism $G \to \Out(N)$ has finite image then $N \vr G$.   
\end{lemma}

\begin{proof} Let $\xi:G \to G/N$ be the quotient homomorphism. 
\Cref{rem:finite_image->Ncent(N)-finite_index} tells us that $\xi(\C_G(N))$ has finite index in $G/N$. By the assumptions,  $G/N$ is virtually free, so we can find a finite index free subgroup  $F' \leqslant_f G/N$ such that $F' \subseteq \xi(\C_G(N))$.
Since free groups are projective, there exists a free subgroup $F \leqslant C_G(N)$ such that the restriction of $\xi$ to $F$ induces an isomorphism between $F$ and $\xi(F)=F'$. It follows that $F \cap N=\{1\}$ and $K \coloneq FN=\xi^{-1}(F')$, so that $|G:K|=|G/N:F'|<\infty$. 

By construction, $F \subseteq \C_G(N)$, so $K$ is naturally isomorphic to the direct product $N \times F$. In particular, $N$ is a retract of $K$, and thus $N \vr G$.
\end{proof}

	\begin{thm}\label{thm:common amalg norm sub}
		Let $G$ be the fundamental group of a finite graph of groups $(\mathcal{G}, \Gamma)$ and let $N \n G$ be a normal subgroup 
		such that all of the following conditions are satisfied:
		\begin{itemize}
			\item[(a)]  $\alpha_e(G_e)=N$ in $G$, for every edge $e \in E\Gamma$;
			\item[(b)] $N$ has a normal virtual complement in $G_{v}$, for each vertex $v \in V\Gamma$;
			\item[(c)] the natural homomorphism $G \to \Out(N)$ has finite image.
		\end{itemize}
		Then $N$ is a virtual retract of $G$.
	\end{thm}
	
	\begin{proof} Choose a maximal tree $T$ in $\Gamma$ and fix an orientation on the edges of $\Gamma$, $E\Gamma=E\Gamma^+ \sqcup E\Gamma^-$, so that $G=\pi_1(\mathcal{G},\Gamma, T,E\Gamma^+)$.
    
	We start by reducing the problem to a simpler graph of groups, where (the image of) $N$ will have finite index in each vertex group.
  For each vertex $v \in V\Gamma$, let $M_v \n G_v$ be a  normal virtual complement to $N$ in $G_v$. Since $M_v \cap \alpha_e(G_e)=\{1\}$, for all $v \in V\Gamma$ and every $e \in E\Gamma$, with $\alpha(e)=v$, we can apply the construction from \Cref{ex:morphism_between_fund_gps-1} to build the quotient graph of groups $(\ot{\mathcal{G}}, \Gamma)$, with the same underlying graph $\Gamma$, the same edge groups $\{G_e\}_{e \in E\Gamma}$ and   with vertex groups $\{\ot{G}_v\coloneq G_v/M_v\}_{v \in V\Gamma}$. We then have a group homomorphism 
  \[\varphi:G \to \ot{G}\coloneq \pi_1(\ot{\mathcal{G}},\Gamma, T,E\Gamma^+),\] whose restriction to $G_v$ is the quotient map $\varphi_v:G_v \to G_v/M_v$, $v \in V\Gamma$.

Observe that $\varphi$ is injective on $N$ and the image of each edge group  of $(\ot{\mathcal{G}}, \Gamma)$ in $\ot{G}$ is $\ot{N}\coloneq \varphi(N)$, which has finite index in each vertex group $\ot{G}_v$ in $\ot{G}$. Thus,  by \Cref{lem: listresultsretr}.(i), it is sufficient to prove that $\ot{N} \vr \ot{G}$. 

Note that $\ot{G}/N$ is naturally isomorphic to the free product of the finite groups $\ot{G}_v/\ot{N}$, $v \in V\Gamma$, with the finite rank free group generated by the free letters (this is true because by taking this quotient we are ``killing'' every edge group of the graph of groups $(\ot{\mathcal{G}}, \Gamma)$), hence it is virtually free. Since the map $\varphi:G \to \ot{G}$ is surjective and $|G:N\C_G(N)|<\infty$, by \Cref{rem:finite_image->Ncent(N)-finite_index}, we see that $|\ot{G}:\ot{N}\C_{\ot{G}}(\ot{N})|<\infty$, whence the natural homomorphism $\ot{G} \to \Out(\ot{N})$ has finite image. Therefore, we can apply \Cref{lem:norm_sbgps_with_virt_free_quot->vr} to conclude that $\ot{N} \vr \ot{G}$, so $N \vr G$, and the proof is complete.
	\end{proof}

\begin{rem}\label{rem:improv_common amalg norm sub} Suppose that in \Cref{thm:common amalg norm sub} we knew that for each $v \in V\Gamma$ either $G_v$ is torsion-free, or $G_v$ is residually finite or $G_v/N$ is residually finite. Then we could replace the hypothesis (a) of this theorem by the weaker assumption that $N\leqslant_f \alpha_e(G_e)$  in $G$, for all $e \in E\Gamma$.
\end{rem}
 
Indeed, under these additional assumptions on $G_v$,  \Cref{prop:nvc->fi_supergp} tells us that for every $v \in V\Gamma$ there is $M_v \n G_v$ that is 
a  normal virtual complement to $\alpha_e(G_e)$, for all $e \in E\Gamma$ with $\alpha(e)=v$.  We can then define $\ot{G}$ and $\ot{N} \n \ot{G}$ as in the proof of \Cref{thm:common amalg norm sub}. The quotient $\ot{G}/\ot{N}$ will now split as the fundamental group of a finite graph of groups with finite vertex groups and (possibly non-trivial) finite edge groups, which is virtually free by \Cref{thm:virtually free}.

\begin{prop}\label{prop:common amalg norm sub}
		Let $G$ be the fundamental group of a finite graph of groups $(\mathcal{G}, \Gamma)$ and let $N \n G$ be a finitely generated virtually abelian normal subgroup 
		such that all of the following conditions are satisfied:
		\begin{itemize}
			\item[(a)]  $\alpha_e(G_e)=N$ in $G$, for every edge $e \in E\Gamma$;
			\item[(b)] $G_v$ has (LR), for all $v \in V\Gamma$;
			\item[(c)] the natural homomorphism $G \to \Out(N)$ has finite image.
		\end{itemize}
		Then $G$ has (LR).
	\end{prop}
    
\begin{proof}
		By \Cref{lem:virtual normal complement}, $N$ has a normal virtual complement in each vertex group $G_v$, so $N \vr G$ by \Cref{thm:common amalg norm sub}.
The first assumption implies that $G/N$ is a finite free product of groups $G_v/N$, $v \in V\Gamma$, with a finitely generated free group $F$. Each free factor has (LR) by \Cref{extensions}.(i) and \Cref{lem: listresultsretr}.(iii), so $G/N$ has (LR) by \Cref{lem: listresultsretr}.(v). Therefore, we can deduce that $G$ has (LR) by applying \Cref{cor:LR extension from quot}.
	\end{proof}

By \Cref{lem: listresultsretr}.(iv), finitely generated virtually abelian groups have (LR), therefore, \Cref{prop:common amalg norm sub} and \Cref{lem:nvr->finite_image} imply \Cref{cor:LR_for_amalgam} from the Introduction. We also have the following analogue for HNN-extensions.

\begin{cor}\label{cor:(LR)_for_HNN} Suppose that $A$ is a finitely generated virtually abelian group and $N \lhd A$ is a normal subgroup with an automorphism $\xi \in \Aut(N)$. Let $G$ be the HNN-extension
\[G=\langle A,t \mid t^{-1}at=\xi(a),\text{ for all } a \in N\rangle.\]
Then $G$ has (LR) if and only if the natural images of $A$ and $\xi$ generate a finite subgroup in $\Out(N)$. In particular, this holds if $A$ is abelian and $\xi$ has finite order.    
\end{cor}

Recall that, given any non-zero integers $k,l$, the 
\emph{Baumslag-Solitar group} $BS(k,l)$ is defined by the one-relator presentation
    \begin{equation}\label{eq:BS(k,l)}
        BS(k,l) = \langle a,t \mid ta^k t^{-1} = a^l \rangle.
    \end{equation}

\begin{cor}\label{cor: BS has (LR)}
    The group $BS(k,l)$  has (LR) if and only if $l  = \pm k$.
\end{cor}

\begin{proof}
    If $BS(k,l)$ has (LR), then it is balanced by \Cref{rem:VRC->balanced} below, so $l= \pm k$. 
    If $l  = \pm k$, then we can deduce that $BS(k,l)$ has (LR) by applying \Cref{cor:(LR)_for_HNN} with $A\coloneq \langle a \rangle \cong \Z$, $N\coloneq \langle a^k \rangle \n A$ and $\xi\in \Aut(N)$ the finite order automorphism defined by $\xi(a^k) = a^l$.    
\end{proof}

\begin{rem}
We do not know whether it is possible to replace  assumption (a) in \Cref{prop:common amalg norm sub} by the weaker condition that $|\alpha_e(G_e):N|<\infty$, for all $e \in E\Gamma$, even in the case when $N=\{1\}$ and $|G_v|<\infty$, for all $v \in V\Gamma$. At the time of writing this paper, it is still an open question whether all finitely generated virtually free groups have (LR), see  \cite[Question~11.1]{virtprops}.    
\end{rem}


\section{Edge-approximating families for fundamental groups of graphs of groups}
\label{sec:edge-approx}
\Cref{thm:edge_approx->cyc_sbgps_are_vr} in this section will provide us with the main method for showing that the fundamental group of a graph of groups has (VRC). This theorem utilizes the following technical tool.

\begin{defn}\label{def: edge_approximates} Let $(\mathcal{G},\Gamma)$ be a graph of groups with fundamental group $G$ and let $\mathfrak{A}$ be a family of groups. In what follows we identify the vertex groups $G_v$, $v \in V\Gamma$, with their images in $G$.

We will say that the family $\mathfrak{A}$ \emph{edge-approximates} $G$ if for every  collection of finite subsets $\{F_v \mid v \in V\Gamma\}$, where $F_v \subseteq G_v\subseteq G$ and $|F_v|<\infty$, for all  $v \in V\Gamma$,  there exists $A \in \mathfrak{A}$ and a homomorphism $\varphi:G \to A$ such that:
\begin{enumerate}[label = (C\arabic*)]
    \item\label{(C1)}  $\varphi(F_{\alpha(e)} \setminus \alpha_e(G_e)) \subseteq A \setminus \varphi(\alpha_e(G_e))$, for every edge $e \in E\Gamma$;  
\item\label{(C2)} the image $\varphi(\alpha_e(G_e))$ is separable in $A$, for each $e \in E\Gamma$.  
\end{enumerate}  
\end{defn}

\begin{lemma}\label{lem: edge aproximates -> finite}
    Let $G$ be the fundamental group of a finite graph of groups $(\mathcal{G},\Gamma)$. If a family $\mathfrak{A}$ of groups edge-approximates $G$, then the family $\mathfrak{A}'$, consisting of finite quotients of groups from $\mathfrak A$, edge-approximates $G$.
\end{lemma}
\begin{proof}
    Suppose that for every $ v \in V\Gamma$ we are given a finite subset  $F_v \subseteq G_v$. Let $A$ be a group in $\mathfrak{A}$ and let $\varphi\colon G \to A$ be a homomorphism satisfying conditions \ref{(C1)} and \ref{(C2)} from \Cref{def: edge_approximates}. For each $e\in E\Gamma$, since $\varphi(F_{\alpha(e)} \setminus \alpha_e(G_e))$ is finite, \Cref{lem:RF+sep} shows that there is a finite index normal subgroup $N_e\n_f A$ such that
    \begin{equation}\label{eq:def_of_N_e}
    \varphi(F_{\alpha(e)} \setminus \alpha_e(G_e))\cap \varphi(\alpha_e(G_e))N_e = \emptyset~\text{ in } A.
    \end{equation}
Since $E\Gamma$ is finite, the subgroup $N\coloneq \bigcap_{e\in E\Gamma} N_e$ has finite index in $A$, hence $P \coloneq A/N \in \mathfrak{A}'$. Equation \eqref{eq:def_of_N_e} shows that
\begin{equation}\label{eq:prop_of_N_in_A}
  \varphi(F_{\alpha(e)} \setminus \alpha_e(G_e))\cap \varphi(\alpha_e(G_e))N = \emptyset~\text{ in } A,~\text{ for all } e \in E\Gamma.    
\end{equation}
Let  $\rho \colon G \to P$ denote the composition of $\varphi$ with the quotient epimorphism $A \to P$. Then  $\rho(F_{\alpha(e)} \setminus \alpha_e(G_e)) \subseteq P \setminus \rho(\alpha_e(G_e))$, for each $e \in E\Gamma$, by \eqref{eq:prop_of_N_in_A}. Since every subgroup of the finite group $P$ is separable, we can conclude that $\mathfrak{A}'$ edge-approximates $G$, as claimed. 
\end{proof}

The next result uses the standard notion of hyperbolic elements in fundamental groups of  graphs of groups, see \Cref{def:ell-hyp}.		

\begin{thm}\label{thm:edge_approx->cyc_sbgps_are_vr}
Let $G$ be the fundamental group of a finite graph of groups $(\mathcal{G},\Gamma)$. If $G$ is edge-approximated by some family of groups $\mathfrak{A}$ then for every hyperbolic element $g\in G$ there is a finitely generated virtually free group $\ot{G}$ and a homomorphism $\varphi:G \to \ot{G}$ such that $\varphi$ is injective on $\langle g \rangle$; in particular, $\langle g \rangle \vr  G $.    
\end{thm}
\begin{proof} Fix a maximal tree $T$ in $\Gamma$ and an orientation $E\Gamma=E\Gamma^+ \sqcup E\Gamma^-$, so that $G=\pi_1(\mathcal{G},\Gamma, T,E\Gamma^+)$.

    Let $ g \in G$ be a hyperbolic element. By \Cref{lem: listresultsretr}.(i) and \Cref{rem:cycl_reduced->powers_are_reduced}, we can replace $g$ by a conjugate, to assume that
    \begin{equation*}
        g = y_1 \cdots y_n,
    \end{equation*}
    where $y_1 \cdots y_n$ is some cyclically reduced expression in $\mathcal{G}$ such that either $n>1$ or $g=t_e^{\pm 1}$, for some $e \in E\Gamma$. In view of \Cref{lem: reduced expression}, we can treat each $y_i$ as an element of $G$. 

Let $F$ be the finite subset of $G$ containing $y_1,\dots,y_n$.
    For every $v \in V\Gamma$, let $F_v$ be the intersection of $F$ with $G_v$ in $G$. Note that if $y_i\in \alpha_e(G_e)$, for some $i \in \{1,\dots,n\}$ and some $e \in ET$, then $y_i \in F_{\alpha(e)} \cap  F_{\omega(e)}$, as it can be seen from presentation \eqref{eq:pres_of_fund_gp}.
    
   According to \Cref{lem: edge aproximates -> finite}, there exists a finite group $K$ and a homomorphism $\rho\colon G \to K$ such that for each edge $e\in E\Gamma$, we have
   \begin{equation}\label{eq: non-edge to non-edge}
       \rho(F_{\alpha(e)} \setminus \alpha_e(G_e)) \subseteq K \setminus \rho(\alpha_e(G_e)).
   \end{equation}

    As in \Cref{ex:morphism of graphs}, we can construct the graph of groups  $(\ot{\mathcal{G}},\Gamma)$ induced by $\rho$ (with finite vertex and edge groups $\{\ot{G}_v\coloneq \rho(G_v)\}_{v \in V\Gamma}$ and $\{ \ot{G}_e=\ot{G}_{\bar{e}}\coloneq \rho(\alpha_e(G_e))\}_{e \in E\Gamma^+}$, respectively), and there is a  homomorphism 
    \begin{equation*}
        \varphi\colon G \longrightarrow \ot{G}\coloneq \pi_1(\ot{\mathcal{G}},\Gamma,T, E\Gamma^+),
    \end{equation*}
    that restricts to $\rho$ on each vertex group and sends free letters $t_e$ in $G$ to free letters $\ot{t}_e$ in $\ot{G}$, for all $e\in E\Gamma^+$.

    Note that $\ot{G}$ is finitely generated as the fundamental group of a finite graphs of groups with finite vertex groups, and it is virtually free, by \Cref{thm:virtually free}. Therefore, $\varphi(\langle g \rangle ) \vr \ot{G}$ by \Cref{lem: listresultsretr}.(vii), so, in view of \Cref{lem: listresultsretr}.(i), it remains to show that $\varphi$ is injective on $\langle g \rangle$.

Recall the equivalence relation on $\bigsqcup_{v \in V\Gamma} G_v$, discussed above \Cref{reduced expressions}.
     Note that if $y_i\in G_v$, for some $i\in \{1,\dots,n\}$ and some $v \in V\Gamma$, and $w \in V\Gamma \setminus \{v\}$ then $y_i$ is not equivalent to an element of $G_w$ in  $\mathcal{G}$ (with respect to $T$) if and only if for some edge $e\in ET$, along the unique path from $v$ to $w$ in $T$, we have that $y_i\in F_{\alpha(e)}\setminus \alpha_e(G_e)$. 
     Equation \eqref{eq: non-edge to non-edge} now implies that $\varphi(y_i) \in \ot{G}_v$ is not equivalent to an element of $\ot{G}_w$ in $\ot{\mathcal{G}}$  (with respect to $T$).
     
     Similarly, by \Cref{lem: reduced expression}, $y_i \in G_v$ is not equivalent to an element in $\alpha_e(G_e)=\omega_{\ov{e}}(G_{\ov{e}})$, for some $e\in E\Gamma$, if and only if either  $y_i$ is not equivalent to an element of $G_{\alpha(e)}$ or $y_i$ is equivalent to some element $x \in G_{\alpha(e)}$ such that $x \notin \alpha_e(G_e)$ in $G_{\alpha(e)}$. The latter means that $y_i=x \in F_{\alpha(e)}\setminus \alpha_e(G_e)$ in $G$. Again, using \eqref{eq: non-edge to non-edge}, we can conclude that $\varphi(y_i)$ is not equivalent to an element of $\ot{\alpha}_e(\ot{G}_e)$ in $\ot{\mathcal{G}}$.
     
     The above observations imply that
     \[\varphi(g) = \varphi(y_{1})\cdots \varphi(y_{n}) \] is a cyclically reduced expression in $(\ot{\mathcal{G}},\Gamma)$ (see Subsection~\ref{subsec:graphs_of_gps}). And it follows that $\varphi(g)$ is also a hyperbolic element in $\ot{\mathcal{G}}$, by \Cref{rem:cycl_reduced->powers_are_reduced}. Since hyperbolic elements have infinite order, we deduce that $\varphi$ is injective on $\langle g \rangle$. Therefore, $\langle g \rangle \vr G$, by \Cref{lem: listresultsretr}.(i), and the proof is complete.
\end{proof}

The proof of \Cref{thm:edge_approx->cyc_sbgps_are_vr} also yields the following result.

\begin{prop}\label{prop:vertex_gps_are_sep}
Suppose that the fundamental group $G$ of a finite graph of groups $(\mathcal{G},\Gamma)$ is edge-approximated by some family of groups $\mathfrak{A}$. Then for every $w \in V\Gamma$, the vertex group $G_w$ is separable in $G$.    
\end{prop}

\begin{proof} As usual, we start by fixing a maximal tree $T$ in $\Gamma$ and an orientation $E\Gamma=E\Gamma^+ \sqcup E\Gamma^-$, and assume that $G=\pi_1(\mathcal{G},\Gamma, T,E\Gamma^+)$.

Choose any $w \in V\Gamma$ and suppose that $g \in G \setminus G_w$. To show that $G_w$ is separable in $G$, we need to find a homomorphism $\psi$ from $G$ to a finite group $Q$ such that $\psi(g) \notin \psi(G_w)$ in $Q$.

Let $g=y_1 \dots y_n$ be a reduced expression for $g$. Then $n \ge 1$ and we can define the finite subsets $F \coloneq \{y_1, \dots,y_n \} \subset G$ and $F_v \coloneq F \cap G_v$, $v \in V\Gamma$. Proceeding as in the proof of \Cref{thm:edge_approx->cyc_sbgps_are_vr}, we construct a graph of groups $(\ot{G}, \Gamma)$ with finite vertex groups $\{\ot{G}_v\}_{v \in V\Gamma}$ and a homomorphism $\varphi:G \to \ot{G}$, where $\ot{G}\coloneq \pi_1(\ot{\mathcal{G}},\Gamma, T,E\Gamma^+) $, such that $\varphi(G_v)=\ot{G}_v$ and $\varphi(t_e)=\ot{t}_e$, for all $v \in V\Gamma$ and $e \in E\Gamma^+$.

The argument in the proof of \Cref{thm:edge_approx->cyc_sbgps_are_vr} implies that $\varphi(y_1)\dots \varphi(y_n)$  is a reduced expression of length $n$ for $\varphi(g)$ in $(\ot{G},\Gamma)$. Thus, if $n \ge 2$ then $\varphi(g) \notin \ot{G}_w=\varphi(G_w)$ in $\ot{G}$, by \Cref{rem:cycl_reduced->powers_are_reduced}. If $n=1$ then either $g=t_e^{\pm 1}$, for some $e \in E\Gamma^+$, or $g \in G_v$, for some $v \in V\Gamma\setminus\{w\}$. In the former case, $\varphi(g) \notin \varphi(G_w)$ as $\varphi(g)=\ot{t}_e^{\pm 1}$ has infinite order, while $\varphi(G_w)$ is finite in $\ot{G}$. In the latter case, $g=y_1 \in G_v$ is not equivalent to an element of $G_w$ in $\mathcal{G}$ (with respect to $T$) by \Cref{lem: reduced expression}. Then the construction of $\varphi$ implies that $\varphi(y_1) \in \ot{G}_v$ is not equivalent to an element of $\ot{G}_w$ in $\ot{\mathcal{G}}$ (with respect to $T$), hence $\varphi(g) \notin \varphi(G_w)$ in $\ot{G}$.

The group $\ot{G}$ is virtually free by \Cref{thm:virtually free}, hence it is residually finite. It follows that the finite subgroup $\varphi(G_w)$ is separable in $\ot{G}$, so there exist a finite group $Q$ and a homomorphism $\xi:\ot{G} \to Q$ such that $\xi(\varphi(g)) \notin \xi(\varphi(G_w))$ in $Q$. Therefore, the composition $\psi\coloneq \xi \circ \varphi:G \to Q$ satisfies the required property and the proof is complete.    
\end{proof}

We will now apply \Cref{thm:edge_approx->cyc_sbgps_are_vr} to deduce several useful corollaries. For the remainder of this section  $G$ will denote the fundamental group of a finite graph of groups $(\mathcal{G}, \Gamma)$. The next corollary establishes \Cref{thm:VRC_for fund_gps_in_intro} mentioned in the Introduction.

\begin{cor}\label{claim 2}
		Let $\varphi\colon G \to A$ be a homomorphism to a group $A$. Suppose that $\varphi$ is injective on each vertex group $ G_v $, $v \in V\Gamma$, and $\varphi(\alpha_e(G_e))$ is separable in $A$, for all edges $e \in E\Gamma$. Then $\langle g \rangle \vr  G $, for every hyperbolic element $g\in G$, and $G_w$ is separable in $G$, for all $w \in V\Gamma$. In particular, if  $A$ has (VRC) then so does $G$.
	\end{cor}
	
	\begin{proof}
            Since $\varphi$ is injective on the vertex groups, the singleton $\{A\}$ edge-approximates $G$, and thus $\langle g \rangle \vr G $, for every hyperbolic element $g\in G$, by \Cref{thm:edge_approx->cyc_sbgps_are_vr}. And each vertex group $G_w$ is separable in $G$, by \Cref{prop:vertex_gps_are_sep}.

Now suppose that $A$ has \VRC. Any elliptic element $h \in G$ is conjugate to an element of a vertex group $G_v$, for some $v \in V\Gamma$, hence the restriction of $\varphi$ to $\langle h \rangle$ is injective. Since $\varphi(\langle h \rangle) \vr A$, because $A$ has (VRC), we can apply \Cref{lem: listresultsretr}.(i) to conclude that $\langle h \rangle \vr G $, as required.	\end{proof}

A special case of \Cref{claim 2} when $A=G$ and $\varphi$ is the identity map deserves to be stated separately. 
\begin{cor}
Suppose that $\alpha_e(G_e)$ is separable in $G$, for each edge $e \in E\Gamma$. Then $G_w$ is separable in $G$, for each $w \in V\Gamma$, and $\langle g \rangle \vr  G $, for every hyperbolic element $g\in G$.
        \end{cor}

	\begin{cor}\label{cor:homom_to_vab->VRC}
 Suppose that all vertex groups in $(\mathcal{G},\Gamma)$ are finitely generated and virtually abelian. Then the following statements are equivalent:
 \begin{itemize}
     \item[(i)]  $G$ has (VRC);
     \item[(ii)] $G_v \vr G$, for each $v \in V\Gamma$;    
     \item[(iii)]  there exists a finitely generated virtually abelian group $A$ and a homomorphism $\varphi\colon G \to A$, such that $\varphi$ is injective on every vertex group $G_v $, $v \in V\Gamma$.
 \end{itemize}
 \end{cor}

\begin{proof} If $G$ has (VRC) then every finitely generated virtually abelian subgroup is a virtual retract by \Cref{lem: props sep}.(ii), so
(i) implies (ii). 
Claim (ii) implies (iii) by \Cref{prop:further_props_of_VRC}.

To show that  (iii) implies (i), recall that   every finitely generated virtually abelian group $A$ has (VRC) by \Cref{lem: listresultsretr}.(iv). Moreover, every subgroup of $A$ is separable, so the result follows from \Cref{claim 2}.
\end{proof}

\begin{rem}\label{rem:algorithm} Condition (iii) of \Cref{cor:homom_to_vab->VRC}
can be restated as 
\begin{itemize} 
    \item[\textit{(iv)}] there is a finite index normal subgroup $N \n_f G$ such that $G_v \cap [N,N]=\{1\}$, for all $v \in V\Gamma$ (in other words, $G_v \cap N$ injects in the abelianization of $N$).
\end{itemize}
One can use this new condition (iv)
to design an algorithm that, starting with a group $G \in \mathfrak C$, enumerates all normal finite index subgroups $N \n_f G$ and checks if $G_v\cap [N,N] = \{1\}$ for all $v\in V\Gamma$. If the result is positive for one of the iterated subgroups $N$, property (VRC) is confirmed for $G$. Unfortunately, this algorithm only works one way: it will stop if and only if $G$ has (VRC), and will run indefinitely  otherwise.
\end{rem}

The implication (ii)$\Rightarrow$(i) in \Cref{cor:homom_to_vab->VRC} is true more generally.

\begin{cor}\label{cor:graph_of_VRC_vertex gps}
Assume that all vertex groups $G_v$ have (VRC) and all edge groups $G_e$ are finitely generated and virtually abelian. If $G_v \vr G$, for each $v \in V\Gamma$,  then $G$ has (VRC).
  \end{cor}  

\begin{proof}
By \Cref{lem:virt_retract->homom}, for each $G_v$, $v \in V\Gamma$, there exist a group $P_v$ satisfying (VRC), and a homomorphism $\varphi_v:G \to P_v$ such that $\varphi_v$ is injective on $G_v$ and $\varphi_v$. Then the direct product $A= \times_{v \in V\Gamma}\, P_v$ has (VRC) (by \Cref{lem: listresultsretr}.(viii)), and the homomorphism
\[\varphi=\times_{v \in V\Gamma} \,\varphi_v:G \to A\]
is injective on $G_v$, for each $v \in V\Gamma$. Moreover, since the edge group $G_e$ is finitely generated and virtually abelian, for each $e \in E\Gamma$, the image $\varphi(\alpha_e(G_e))$ is a virtual retract of $A$, hence it is separable in $A$ by \Cref{lem: props sep}. Therefore, we can apply \Cref{claim 2} to conclude that
$G$ has (VRC).
\end{proof}

\begin{cor}\label{cor:special_HNN_has_VRC}
Let $A$ be a group with (VRC) and let $B \leqslant A$ be a separable subgroup of $A$. Then the HNN-extension 
\begin{equation}\label{eq:spec_HNN}
 G \coloneq \langle A,t \mid  tbt^{-1}=b,~\text{for all } b \in B\rangle   
\end{equation}
has (VRC).    
\end{cor}

\begin{proof} Note that there is a natural retraction $\varphi:G \to A$ such that \begin{equation}\label{eq:nat_retr}
 \varphi(a)=a, \text{ for all }a \in A, ~\text{ and }~\varphi(t)=1.    
\end{equation}
Therefore, the claim follows from \Cref{claim 2}.   \end{proof}

\Cref{cor:special_HNN_has_VRC} can be used to construct (VRC) groups with unexpected properties. We demonstrate two instances of this in the next examples.

\begin{ex}\label{ex:VRC_with_unsolavble_WP} There exists a finitely generated recursively presented group $G$ with (VRC) that has unsolvable word problem.

Indeed, let $P$ be the $2$-generated residually finite recursively presented group with unsolvable word problem constructed by Meskin in \cite{Meskin}. Then $P \cong A/B$, where $A$ is the free group of rank $2$ and $B$ is a separable (by \Cref{rem:normal_sbgp_sep<->quot_rf}) recursively enumerable normal subgroup of $A$ such that the membership problem to $B$ in $A$ is undecidable. 

Let $G$ be the HNN-extension \eqref{eq:spec_HNN}. Then $G$ is $3$-generated, recursively presented and has (VRC) by \Cref{lem: listresultsretr}.(iii) and \Cref{cor:special_HNN_has_VRC}. However, there is no algorithm deciding whether $[t,a]=1$ in $G$, for elements $a \in A$, because such an algorithm would solve the membership problem to $B$ in $A$. Therefore, $G$ has unsolvable word problem.    
\end{ex}

Recall that a group $H$ is said to be \emph{residually solvable} if for every $h \in H \setminus\{1\}$ there is a solvable group $M$ and a homomorphism $\psi:H \to M$ such that $\psi(h) \neq 1$ in $M$. It is well-known that RFRS groups are residually solvable (see, for example, \cite[Proposition~4.2]{Kielak}). By \cite[Chapter~3]{Droms-thesis}, right angled Artin groups are residually (torsion-free nilpotent), hence they are residually solvable. Therefore, any virtually special group is virtually residually solvable. (It is also true that any virtually special group is virtually RFRS, see \cite[Corollary~2.3]{Agol}).

\begin{ex} \label{ex:VRC_not_virt_RFRS} There exists a finitely generated free-by-free group $G$ that has (VRC) and no finite index subgroup of $G$ is residually solvable. Thus $G$ is locally indicable but it is neither virtually RFRS nor virtually special.

Let $S$ be the (restricted) direct product of all finite alternating groups. Then $S$ is countable and residually finite, so it can be embedded in a finitely generated residually finite group $P$, by a theorem of Wilson \cite{Wilson}. Evidently $P$ is not virtually residually solvable, as every finite index subgroup contains a simple alternating group that is not residually solvable. 

Let $A$ be a finite rank free group with 
an epimorphism $\xi:A \to P$. Set $B \coloneq \ker\xi \n A$ and let $G$ be the HNN-extension \eqref{eq:spec_HNN}. 
Since $P$ is residually finite, from \Cref{rem:normal_sbgp_sep<->quot_rf} we know that $B$ is separable in $A$, whence $G$ has (VRC) by \Cref{cor:special_HNN_has_VRC}. Note that the kernel of the retraction $\varphi:G \to A$, given by \eqref{eq:nat_retr}, is free, as it intersects  trivially the base $A$, of the HNN-extension $G$,  therefore $G$ is free-by-free. It remains to show that no finite index subgroup of $G$ is residually solvable. 

Consider any $K \leqslant_f G$. 
Then $\xi(A \cap K) \leqslant_f P$ is not residually solvable, so there is an element $y \in \xi(A \cap K) \setminus\{1\}$ such that the image of $y$ is trivial in every solvable quotient of $\xi(A \cap K) \cong (A \cap K)/(B \cap K)$. 
Choose any $x \in A \cap K$ such that $\xi(x)=y$. Then $x\notin B$ but for every homomorphism $\eta:A \cap K \to L$, where $L$ is a solvable group, we must have 
\begin{equation}\label{eq:x_in_eta}
  \eta(x) \in \eta(B \cap K)~\text{ in } L.  
\end{equation}

Take any $m \in \N$ such that $t^m \in K$, and note that $[x,t^m] \neq 1$ in $G$, by Britton's lemma for HNN-extensions. 
Assume, by contradiction, that $K$ is residually solvable. Then there is solvable group $M$ and an epimorphism $\psi:K \to M$ such that $\psi([x,t^m]) \neq 1$ in $M$. Since $B\cap K \subseteq \C_K(t^m)$, we can conclude that $\psi(x)  \notin \psi(B \cap K)$ in $M$. The latter clearly contradicts \eqref{eq:x_in_eta}. This shows that $K$ cannot be residually solvable, as  required.    
\end{ex}

    
\section{Trees of abelian  groups}\label{sec:trees_of_ab_gps}
In this section we will use central products to show that fundamental groups of finite trees of finitely generated abelian groups have (VRC).

We start by recalling the notion of central product (cf. \cite[Section~2]{Neu-Neu}, where these are called \emph{generalized direct products with amalgamation}), which can be used to define pushouts in the category of abelian groups.

\begin{defn}\label{centralprod}
		Let $A$ and $B$ be two groups. And let $H$ be another group along with two embeddings $\iota_1\colon H \hookrightarrow A$, $\iota_2\colon H \hookrightarrow B$. Assume that the images of $H$ are central in $A$ and in $B$. We define 		
		the \emph{central product} $A\times_H B$ as the quotient
		\begin{equation*}
			A\times_H B = \frac{A\times B}{  \{(\iota_1(h),\iota_2(h)^{-1})\mid h\in H\} }.
		\end{equation*}
Note that  $\left\{\left(\iota_1(h),\iota_2(h)^{-1} \right)\mid h\in H\right\}$ is a central subgroup in the direct product $A \times B$ because the images of $H$ in $A$ and $B$ are central.  To simplify notation, we will sometimes omit the maps $\iota_1$ and $\iota_2$, treating $H$ as a subgroup of both $A$ and $B$.

\begin{rem}\label{rem:prop_of_central_prod} Let $\xi:A \times B \to A \times_H B$ be the quotient map. After identifying $A$ and $B$ with the subgroups    $(A,1)$ and $(1,B)$    of $A \times B$, we see that 
\begin{itemize}
    \item[(i)] $\xi$ is injective on $A$ and on $B$;
    \item[(ii)] $\xi(A) \cap \xi(B)=\iota(H)$, where $\iota=\xi \circ \iota_1=\xi\circ\iota_2:H \to A \times_H B$.
\end{itemize}
\end{rem}
	
	\end{defn}
	\begin{rem}\label{rem:hom_from_amal_free_prod_to_central_prod}
		Let $A$, $B$ and $H$ be as in the \Cref{centralprod}. Then $ A\times_H B $ is naturally isomorphic to the quotient of the amalgamated free product $G \coloneq A*_H B$ by $[A,B]$, where $[A,B]$ is the normal closure of the set  $\{[a,b] \mid a \in A,~b \in B\}$, consisting  of all commutators of elements from $A$ with elements from $B$. Moreover, \Cref{rem:prop_of_central_prod}  tells us that the quotient map $G \to A \times_H B$ is injective on the union $A \cup B$ in $G$.

        In particular, if $A$ and $B$ are abelian then $A \times_H B$ is simply the abelia\-nization of $ A*_H B$.
	\end{rem}
	
	\begin{lemma}\label{lem: graph of groups - tree - hom to v.a}
		Let $G$ be the fundamental group of a graph of groups $(\mathcal{G},\Gamma)$, in which the underlying graph $\Gamma$ is a finite tree and each vertex group $G_v$ is abelian. Let $\varphi \colon G \to P$ denote the natural homomorphism from $G$ to its abelianization $P=G/[G,G]$. Then   $\varphi$ is injective on the union of all vertex groups $\bigcup_{v \in V\Gamma} G_v$ in $G$.
	\end{lemma}
    
\begin{proof} Note that since $\Gamma$ is a tree, it is its own maximal tree. Choose any orientation $E\Gamma=E\Gamma^+ \sqcup E\Gamma^-$.
    
We proceed by induction on $|E\Gamma|$ in $\Gamma$. If $n=0$ then the statement is trivial, so suppose that $|E\Gamma|>0$. Then $\Gamma$ contains a leaf $u \in V\Gamma$, so that there is only one edge $e \in E\Gamma^+$ incident to $u$ in $\Gamma$. Let $\Gamma'$ be the graph obtained from $\Gamma$ by removing $u$, $e$ and $\ov{e}$. Then we can define a new tree of groups $(\mathcal{G}',\Gamma')$ as the restriction of $(\mathcal{G},\Gamma)$ to $\Gamma'$. After defining the orientation on $E\Gamma'$ by ${E\Gamma'}^+=E\Gamma' \cap E\Gamma^+$, we see that $G=\pi_1(\mathcal{G}, \Gamma, \Gamma,E\Gamma^+)$ decomposes as the amalgamated free product
		\begin{equation*}
			G = G' *_{G_e} G_u,
		\end{equation*}
where $G'=\pi_1(\mathcal{G}', \Gamma', \Gamma',{E\Gamma'}^+)$.

By the induction hypothesis, the abelianization map \[\varphi':G' \to P' \coloneq G'/[G',G']\] is injective on the union of the vertex groups $\bigcup_{v \in V\Gamma'} G_v$ in $G'$. In particular, this map is injective on $G_e$, whence we have a homomorphism
\[\psi:G \to H \coloneq  P'*_{G_e} G_u,\] such that $\psi$ restricts to $\varphi'$ on $P'$ and to the identity map on $G_u$.

Since $P'$ and $G_u$ are abelian groups, \Cref{rem:hom_from_amal_free_prod_to_central_prod} tells us that the abelianization map
$\eta:H \to P \coloneq H/[H,H]$ is injective on $P' \cup G_u$ in $H$. Therefore, the composition $\varphi \coloneq \eta \circ \psi: G \to P$ is injective on the union $\bigcup_{v \in V\Gamma} G_v$. It is easy to see that $P \cong G/[G,G]$ is the abelianization of $G$ and $\varphi$ is exactly the abelianization map $G \to P$.
\end{proof}

We are now ready to prove the main result of this section.
	\begin{thm}\label{thm:tree of abelian->VRC}
		Let $G$ be the fundamental group of a finite graph of groups $(\mathcal{G}, \Gamma)$, in which the underlying graph $\Gamma$ is a tree and each vertex group is finitely generated abelian. Then $G$ has (VRC).
	\end{thm}
    
\begin{proof}
	   By \Cref{lem: graph of groups - tree - hom to v.a}, the abelianization homomorphism $\varphi\colon G \to P\coloneq G/[G,G]$ is injective on each vertex group $G_v$, $v \in V\Gamma$. Since all vertex groups are finitely generated and the graph $\Gamma$ is finite, $G$ is also finitely generated, thus $P$ is a finitely generated abelian group. Therefore, we can use \Cref{cor:homom_to_vab->VRC} to conclude that $G$ has (VRC), as claimed.
	\end{proof}

An amalgamated free product of two finitely generated abelian groups always has (LR), by \Cref{cor:LR_for_amalgam}. However this may fail for larger trees of abelian groups, as 
     demonstrated  in the following example.
    
    \begin{ex}\label{RAAG not Lerf}
         Let $L$ be the right angled Artin group corresponding to a path of length 3, so \[L=\langle a,b,c,d \mid [a,b]=[b,c]=[c,d]=1\rangle.\] Then $L$ can be  expressed as the fundamental group of the tree of groups depicted in  \Cref{fig:path-3}, with three vertex groups that are free abelian of rank $2$ and two edge groups that are infinite cyclic.
        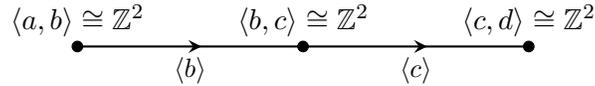
\begin{figure}[H]
            \centering
            \begin{tikzpicture}
    		\draw[fill=black] (-2,0) circle (2pt) node[](0){} node[above]{$\langle a,b \rangle\cong \mathbb{Z}^2  $};
    		\draw[fill=black] (1,0) circle (2pt) node[](1){} node[above]{$\langle b,c \rangle\cong \mathbb{Z}^2$};
    		\draw[fill=black] (4,0) circle (2pt) node[](2){} node[above]{$\langle c,d \rangle \cong \mathbb{Z}^2$};
    		\path [thick,draw=black,postaction={on each segment={mid arrowa={black,scale=1.2}{0.55}}}]
                
                (0.center) -- node[below = 0.1pt]{\small$\langle b \rangle$}  (1.center) --
                 node[below = 0.1pt]{\small$\langle c \rangle$}  (2.center);
    	\end{tikzpicture}
        \caption{Graph of groups for the group $L$ in \Cref{RAAG not Lerf}}\label{fig:path-3}
        \end{figure}
        By \cite[Theorem~1.2]{Niblo-Wise_Raag}, the group $L$ is not LERF, which implies, in view of \Cref{lem: props sep}.(iii), that $L$  does not have (LR).
    \end{ex}


\section{(VRC) and actions on $\R^n$} \label{sec:action_on_R^n}
In this section we obtain a ``geometric'' criterion for property (VRC) based on the following useful folklore result, versions of which are usually discussed with Bieberbach theorems from the theory of crystallographic groups. Recall that a \emph{lattice} in $\R^n$ is a discrete subgroup $L \leqslant \R^n$ such that the quotient $\R^n/L$ is compact. In such a case we necessarily have $L \cong \Z^n$.

\begin{prop}\label{prop:virt_ab_emb_into_virt_Rn} Suppose that $P$ is a group with a  normal subgroup $A \lhd P$ such that $A \cong \Z^n$, for some $n \in \N_0$, and $Q\coloneq P/A$ is finite. Then $P$ can be embedded into a semidirect product $\R^n \rtimes Q$, where the image of $A$ is a lattice in $\R^n$ and $Q$ acts on $\R^n$ by Euclidean isometries, via some (possibly non-faithful) linear representation $Q \to O(n)$. 
\end{prop}

\begin{proof}
The argument below essentially repeats the proof of Zassenhaus' theorem from \cite[Theorem~2.2 in Section~2.2]{Szczep}.  

Since $A$ is a normal subgroup of $P$, $P$ acts on $A$ by conjugation. As $A$ is abelian, this induces an action of $Q$ on $A$, resulting in a homomorphism $\varphi:Q \to \Aut(A) \cong \mathrm{GL}(n,\Z)$. Now, any isomorphism $A \to \Z^n$ gives rise to an embedding $\iota:A \to \R^n$ so that $\iota(A)$ is a lattice in $\R^n$, and since $\mathrm{GL}(n,\Z)$ is a subgroup of $\mathrm{GL}(n,\R)$, we can extend the action of $Q$ on $A$ to a linear action of $Q$ on $\R^n$, so that the embedding $\iota$ becomes a homomorphism of $Q$-modules. By \cite[Exercise 1(b) in Section~IV.4.3]{Brown}, there exist a group $P'$ and a homomorphism $\iota':P \to P'$ completing the following commutative diagram, where the horizontal maps form short exact sequences:
\begin{equation}\label{eq:Brown}
\begin{tikzcd}
\{0\} \arrow[r] & A \arrow[r] \arrow[d, "\iota"', hook] & P \arrow[r] \arrow[d, "\iota'"] & Q \arrow[d, equal] \arrow[r] & \{1\} \\
\{0\} \arrow[r] & \R^n \arrow[r]                        & P' \arrow[r]                    & Q \arrow[r]                                & \{1\}
\end{tikzcd}    
\end{equation}
Since the left and the right vertical maps are injective, so is the homomorphism $\iota'$. Now, the short exact sequence at the bottom of \eqref{eq:Brown} splits because $H^2(Q,\R^n)=\{0\}$ (see \cite[Corollary~III.10.2]{Brown}), so $P' \cong \R^n \rtimes Q$, where $Q$ acts on $\R^n$ via the above homomorphism $Q \to \mathrm{GL}(n,\R)$ (not necessarily by isometries). Finally, as $Q$ is finite, we can find a scalar product on $\R^n$ which is invariant under the action of $Q$ (see, for example, \cite[Exercise~2.15 in Section~2.4]{Szczep}). Once we equip $\R^n$ with the Euclidean metric coming from this scalar product, the action of $Q$ on $\R^n$ will be by isometries, so the proof of the   proposition is complete.    
\end{proof}

\begin{defn}\label{def:Eucl-by-fin} Let $P$ be any semidirect product $\R^n \rtimes Q$, where $n \in \N_0$ and $Q$ is a finite group acting on $\R^n$ linearly by Euclidean isometries (i.e., via some representation $Q \to \mathrm{O}(n)$). Then we will say that $P$ is a \emph{Euclidean-by-finite group of dimension $n$}  (or simply a \emph{Euclidean-by-finite group} if the dimension is not important).

A subgroup $M$ of a Euclidean-by-finite group $P=\R^n \rtimes Q$ will be called \emph{discrete} if $M \cap \R^n$ is a discrete subgroup of $\R^n$ (here and later we abuse the notation by identifying $\R^n$ with $(\R^n,1) \leqslant P$).
\end{defn}

\begin{rem} \label{rem:E-by-f_acts_by_isoms}
Any Euclidean-by-finite group $P=\R^n \rtimes Q$ admits a natural action on $\R^n$ by affine Euclidean isometries. More precisely, if $(\vec u,q) \in P$, where $\vec u \in \R^n$ and $q \in Q$, then we define 
\[(\vec u,q) . \vec x=q.\vec x+\vec u,~\text{ for all }\vec x \in \R^n,\] where $q.\vec x=q \vec x q^{-1} \in \R^n$ in $P$. In particular, the subgroup $\R^n \leqslant P$ acts by translations and $Q$ acts linearly, preserving the origin.  
The kernel $N \n P$ of this  action is necessarily contained in $Q$ and $P/N \cong \R^n \rtimes (Q/N)$.
\end{rem}

We are now in a position to prove \Cref{thm:hom_to_Rn_semidir_fin} from the Introduction.

\begin{proof}[Proof of \Cref{thm:hom_to_Rn_semidir_fin}] Assume that $G$ has (VRC). By \Cref{prop:further_props_of_VRC}, there exists a finitely generated virtually abelian group $P$ and a homomorphism $\psi:G \to P$ such that $\psi$ is injective on each vertex group of $\mathcal G$. Claim (ii) now follows from \Cref{prop:virt_ab_emb_into_virt_Rn}.

Clearly (ii) implies (iii), so it remains to show that (iii) implies (i). This is a consequence of \Cref{cor:homom_to_vab->VRC} and the observation that  since $G$ is finitely generated then so is its image in the virtually abelian group $\R^n \rtimes Q$.
\end{proof}

We will now apply \Cref{thm:hom_to_Rn_semidir_fin} to characterize property (VRC) for some double HNN-extensions of $\Z^2$ with cyclic edge groups.

    \begin{prop}\label{prop:G_k}
        For any $k \in \Z$, let $G_k$ be the group defined by
    \begin{equation}\label{eq:pres_of_G_k}
        G_k \coloneq \langle a,b,s,t \mid [a,b]=1,~sa s^{-1} = b,~t b t^{-1} = a b^k \rangle.
    \end{equation}
    Then $G_k$ has (VRC) if and only if $|k|\le 1$.
    \end{prop}
    
\begin{proof}
Note that $G_k$ is a double HNN-extension of the free abelian group $\langle a,b \rangle \cong \Z^2$. Suppose that this group has (VRC). Then, according to \Cref{thm:hom_to_Rn_semidir_fin}, we have a homomorphism $\varphi:G_k \to P$, for some Euclidean-by-finite group $P=\R^n \rtimes Q$, such that $\varphi$ is injective on $\langle a,b \rangle$.

Set $m =|Q| \in \N$, then $\varphi(a^m)=(\vec u, 1) \in (\R^n,1)$, $ \varphi(b^m)=(\vec v,1) \in (\R^n,1)$ in $P$. The group $P$ acts on its normal subgroup $(\R^n,1)$ by conjugation (thus $(\R^n,1)$ acts trivially and $Q$ acts linearly,  preserving a Euclidean norm $\|\cdot\|$). The relations from the presentation \eqref{eq:pres_of_G_k} of $G_k$ imply that $sa^m s^{-1}=b^m$ and $tb^mt^{-1}=a^{m} b^{mk}$ in $G_k$. Therefore,
\[\varphi(s). \vec{u}=\vec{v} \text{ and  } \varphi(t). \vec{v} =\vec{u}+k\vec{v}~\text{ in } \R^n,\]
hence we must have that 
\begin{equation}\label{eq:same_norms}
\|\vec u\|=\|\vec v\|=\|\vec u+k\vec v\|>0.
\end{equation}

It follows that 
\begin{multline*}
 \|\vec u\|^2=\|\vec u+k\vec v\|^2=\|\vec u\|^2+2k\|\vec u \|\,\|\vec v\|\cos\alpha+ k^2\|\vec v \|^2  \\=\|\vec u\|^2(1+ 2k \cos \alpha+k^2),   
\end{multline*}

where $\alpha \in [0,\pi]$ is the Euclidean angle between $\vec u$ and $\vec v$. Hence,
\begin{equation}\label{eq:cosine}  
\text{either } k=0 \text{ or } \cos\alpha=-\frac{k}{2},   
\end{equation}
so we must have $|k| \le 2$. Note that if $k=\pm 2$ then $\alpha \in \{0,\pi\}$, which would mean that $\vec v= \mp \vec u$ in $\R^n$, contradicting the assumption that $\varphi$ is injective on $\langle a,b \rangle$. Therefore, we can conclude that $|k|\le 1$.
        
We will now show that  the group $G_k$ has (VRC) when $|k| \le 1$, by producing an explicit homomorphism $\varphi:G_k \to P$ that is injective on the base group $\langle a,b\rangle$, and where $P$ is some Euclidean-by-finite group of dimension $2$. If we assume that $\varphi(a)=\vec u \in \R^2$ and $\varphi(b)=\vec v \in \R^2$, then Equation \eqref{eq:same_norms} tells us that $\vec u$ and $\vec v$ must have the same norms, and  \eqref{eq:cosine} gives us the angle between these vectors.

Thus, if $k=0$ we set $P_0 \coloneq \R^2 \rtimes \langle c \rangle_2$, where $c$ acts on $\R^2$ as the reflection with the matrix $\begin{pmatrix} 0 & 1 \\ 1& 0  
\end{pmatrix}$. We define a set map $\ov{\varphi}_0:\{a,b,s,t\} \to P_0$ by 
\[\ov{\varphi}_0(a)=\begin{pmatrix} 1\\ 0\end{pmatrix} \in \R^2,~\ov{\varphi}_0(b)=\begin{pmatrix} 0\\ 1\end{pmatrix} \in \R^2~\text{ and }~ \ov{\varphi}_0(s)=\ov{\varphi}_0(t)=c\]
(to simplify the notation we have identified $\R^2$ with the subgroup $(\R^2,1)$ of $ P_0$).
Since the action of $c$ interchanges $\ov{\varphi}_0(a)$ with $\ov{\varphi}_0(b)$, we see that $\ov{\varphi}_0$ extends to a unique group homomorphism $\varphi_0:G_0 \to P_0$ and this homomorphism is obviously injective on $\langle a , b \rangle$.

When $k \in \{\pm 1\}$, we choose $P\coloneq \R^2 \rtimes D_6$, where $D_6$ is the dihedral group of order $12$, generated by the reflection $d$, in the line passing through the origin and parallel to the vector $\begin{pmatrix} \sqrt{3}/2\\ 1/2\end{pmatrix}$, and by 
the rotation $r$, about the origin by $\pi/3$  anti-clockwise. In other words $d$ and $r$ act on $\R^2$ linearly, via the matrices 
$\begin{pmatrix} 1/2 & \sqrt{3}/2 \\ \sqrt{3}/2& -1/2\end{pmatrix}$ and
$\begin{pmatrix} 1/2 & -\sqrt{3}/2 \\ \sqrt{3}/2& 1/2\end{pmatrix}$ respectively. 

For $k=-1$, we define $\varphi_{-1}:G_{-1} \to P$ by
\[\varphi_{-1}(a)=\begin{pmatrix} 1\\ 0\end{pmatrix},~\varphi_{-1}(b)=\begin{pmatrix} 1/2 \\ \sqrt{3}/2\end{pmatrix},~ \varphi_{-1}(s)=d ~\text{ and }~ \varphi_{-1}(t)=r^{-2}.\]
Again, it is easy to see that $\varphi_{-1}$ is a well-defined homomorphism and it is injective on $\langle a,b \rangle$.

Finally, for $k=1$ we define $\varphi_{1}:G_{1} \to P$ by
\[\varphi_{1}(a)=\begin{pmatrix} 1\\ 0\end{pmatrix},~\varphi_{1}(b)=\begin{pmatrix} -1/2 \\ \sqrt{3}/2\end{pmatrix},~ \varphi_{1}(s)=dr^{-1}~\text{ and }~\varphi_{1}(t)=r^{-1}.\]
Therefore, $G_k$ has (VRC) when $|k| \le 1$, by \Cref{thm:hom_to_Rn_semidir_fin}.
\end{proof}

\begin{rem}\label{rem:homoms_from_G_k} The proof of \Cref{prop:G_k} shows that when $|k| \ge 3$, in every homomorphism $\varphi:G_k \to P$, where $P$ is a finitely generated virtually abelian group (which can be embedded in a Euclidean-by-finite group by \Cref{prop:virt_ab_emb_into_virt_Rn}), we must have $\varphi(a^m) = \varphi(b^m) = 1$, for some $m\in \N$. In particular, $\varphi(a)$ and $\varphi(b)$ must have finite order.

Similarly, if $|k|=2$ then for every homomorphism $\varphi$ from $G_k$ to a virtually abelian group $P$ there will exist $m \in \N$ such that $\varphi(b^m)=\varphi(a^{\pm m})$. 
\end{rem}

\begin{ex}\label{ex:H_k} Let $A$ be the virtually abelian group $\Z \wr C_2$, thus 
\[A=\langle a,b,s \mid [a,b]=1,~s^2=1,~sas^{-1}=b,~sbs^{-1}=a \rangle.\]
For any $k \in \Z$, consider the HNN-extension of $A$ given by 
   \begin{equation}\label{eq:pres_of_H_k}
        H_k =\langle A,t \mid t b t^{-1} = a b^k \rangle .
    \end{equation}
Then $H_k$ has (VRC) if and only if $|k|\le 1$.

To see this, note that 
\[H_k \cong \langle a,b,s,t \mid [a,b]=1,~s^2=1,~sa s^{-1} = b,~t b t^{-1} = a b^k \rangle,\]
so this group is isomorphic to the quotient of the group $G_k$, from \Cref{prop:G_k}, by the normal closure of $s^2$. If $H_k$ has (VRC) then there is a homomorphism from $H_k$ to a virtually abelian group $P$ that is injective on the base group $A$, which implies that $G_k$ has (VRC) by \Cref{cor:homom_to_vab->VRC}. Therefore, according to \Cref{prop:G_k}, we must have $|k|\le 1$.

If $|k|\le 1$ then the homomorphisms from $G_k$ to the Euclidean-by-finite groups $P_0$ and $P$, constructed in \Cref{prop:G_k}, factor through $H_k$ because each time the element $s$ is sent to a reflection. So, when $|k|\le 1$, the group $H_k$ has (VRC) by \Cref{thm:hom_to_Rn_semidir_fin}.
\end{ex}

\begin{ex} \label{ex:G_kl}
For any pair $(k,l) \in \Z^2 \setminus\{(0,0)\}$, let us consider  the double HNN-extension \eqref{eq:G_kl}, mentioned in the Introduction.
Thus $G_{k,1}$ is the group $G_k$ from \Cref{prop:G_k}.
Using the same methods as in that proposition,  we can show that if $G_{k,l}$ has (VRC) then either $k=0$ and $l= \pm 1$ or  $l=0$ and $k= \pm 1$ or $l = \pm k$. In the cases when $k,l \in \{0,\pm 1\}$, it is easy to verify that $G_{k,l}$ has (VRC) similarly to \Cref{prop:G_k}. On the other hand, if $|k|=|l| \ge 2$ then the map 
\[a \mapsto a^k, ~b \mapsto b^k,~s \mapsto s,~t \mapsto t\] gives rise to a surjective but non-injective endomorphism from $G_{k,l}$ to itself, hence this group is non-Hopfian (in fact, the groups $G_{k,k}$, for $k \ge 2$, are isomorphic to the non-Hopfian CAT($0$) tubular groups discovered by Wise in \cite{Wise-non-Hopf}, see also \cite[Example~III.$\Gamma$.7.6]{Bridson_Haefliger}). Therefore, in this case $G_{k,l}$ is not residually finite, so it does not have (VRC) by \Cref{lem: props sep}.(ii). 

We thus see that $G_{k,l}$ has (VRC) if and only if $k,l \in \{0, \pm 1\}$.
\end{ex}

\begin{rem} Note that the associated subgroups $\langle a \rangle$, $\langle b \rangle$ and $\langle a^{-1}b^2 \rangle$ are mapped injectively into the abelianization of the double HNN-extension $G_{2,-1}$, given by \eqref{eq:G_kl}. However, this group does not have (VRC) by \Cref{ex:G_kl}. Thus the assumption that the homomorphism is injective on all \emph{vertex} groups in \Cref{cor:homom_to_vab->VRC} and \Cref{thm:hom_to_Rn_semidir_fin} cannot be replaced by the weaker condition that it is only injective on the edge groups.    
\end{rem}

\begin{rem}\label{rem:signed_perm}
A convenient source of finite subgroups in $\mathrm{GL}(n,\R)$ is provided by signed permutations of a basis. If we fix any basis $\{\vec e_1,\dots,\vec e_n\}$ of $\R^n$ then the \emph{group of signed permutations} $Q \leqslant \mathrm{GL}(n,\R)$ (with respect to this basis) consists of all invertible linear transformations of $\R^n$ sending each $\vec e_i$ to $\pm \vec e_j$, for some $j=j(i) \in \{1,\dots,n\}$. The order of this subgroup $Q$ is $2^n \cdot n!$.
\end{rem}

\Cref{thm:hom_to_Rn_semidir_fin} allows us to reverse engineer graphs of abelian groups whose fundamental groups have (VRC). 
We demonstrate this in the following example, which is based on taking the $(a,b)$-plane to be the plane with equation $x+y+2z=0$ in $\R^3$. This example shows that the dimension of the Euclidean-by-finite group witnessing the fact that the fundamental group $G$ has (VRC) in  \Cref{thm:hom_to_Rn_semidir_fin} may be greater than the largest rank of an  abelian subgroup in the vertex groups of $(\mathcal{G},\Gamma)$.

\begin{ex}\label{ex:3-dim}
Let $G$ be the double HNN-extension of $\langle a,b \rangle \cong \Z^2$ defined by the following presentation:
\begin{equation}\label{eq:pres_of_G}
        G =\langle a,b,s,t \mid [a,b]=1,~sa s^{-1} = b,~t a^5 b t^{-1} = a^{-3} b^5 \rangle.
    \end{equation}
We claim that $G$ has (VRC) and if $\varphi:G \to \R^n \rtimes Q$ is a homomorphism that is injective on $\langle a,b \rangle$, where $Q$ is a finite group acting on $\R^n$ linearly,  then $n \ge 3$.

To show that $G$ has (VRC), let $R \leqslant \mathrm{O}(3)$ be the group of all signed permutation matrices, so that $|R|=2^3\cdot 3!=48$. We define a set map $\ov{\psi}: \{a,b,s,t\} \to \R^3 \rtimes R$ as follows. Let
\[\ov{\psi}(a)=\begin{pmatrix} 2 \\0\\ -1\end{pmatrix} \in \R^3,~\ov{\psi}(b)=\begin{pmatrix} 0 \\ 2 \\-1\end{pmatrix} \in \R^3 ,\] \[\ov{\psi}(s)=\begin{pmatrix}
 0 & 1 & 0 \\ 1& 0& 0 \\ 0&0&1   
\end{pmatrix} \in R~\text{ and }~\ov{\psi}(t)=\begin{pmatrix}
 0 & 0 & 1 \\ 1& 0& 0 \\ 0&-1&0   
\end{pmatrix} \in R.\]

One readily checks that \[\ov{\psi}(s).\ov{\psi}(a)=\ov{\psi}(b) \text{  and } \ov{\psi}(t).(5\ov{\psi}(a)+\ov{\psi}(b))=-3\ov{\psi}(a)+5\ov{\psi}(b) \text{ in }\R^3,\] so $\ov{\psi}$ extends to a unique group homomorphism $\psi:G \to \R^3 \rtimes R$. Evidently, $\psi$ injective on $\langle a,b \rangle$, so $G$ has (VRC) by \Cref{thm:hom_to_Rn_semidir_fin}. 

Now, arguing by contradiction, suppose that for some $n \le 2$ we have a homomorphism $\varphi:G \to \R^n \rtimes Q$  that is injective on $\langle a,b \rangle$, where $Q$ is a finite group acting on $\R^n$ via some (possibly non-faithful) representation $Q \to \mathrm{GL}(n,\R)$. 

It is easy to see that we must have $n=2$.
We can change the basis of $\R^2$ to assume that the image of $Q$ under this representation is contained in $\mathrm{O}(2)$, i.e., $Q$ acts by Euclidean isometries fixing the origin (see \cite[Exercises~2.14 and 2.15 in Section~2.4]{Szczep}).

 Let $m=|Q|$, so that 
\[\varphi(a^m)=(\vec u,1) \in (\R^2,1)~\text{ and } \varphi(b^m)=(\vec v,1) \in (\R^2,1).\]
In view of the presentation \eqref{eq:pres_of_G}, we see that 
\begin{equation*}
  \|\vec u\|=\|\vec v\|  ~\text{ and }~ \|5\vec u+\vec v\|=\|-3\Vec u+5\Vec v\|.
\end{equation*}
As in the proof of \Cref{prop:G_k}, we can use these equations to calculate that the angle $\alpha$ between $\vec u$ and $\vec v$ in $\R^2$ must satisfy $\cos\alpha=1/5$. So, after scaling the Euclidean metric on $\R^2$ and applying a rotation/reflection, we can suppose that 
\begin{equation}\label{eq:coords_of_u_and_v}
 \vec u   =\begin{pmatrix} 1 \\0\end{pmatrix} \text{ and }
 \vec v= \begin{pmatrix} 1/5 \\\sqrt{24}/5\end{pmatrix}.
\end{equation}

Set $\vec x\coloneq 5 \vec u+\vec v$ and $\vec y \coloneq -3\vec u+5\vec v$ in $\R^2$. Let $\beta$ be the angle between $\vec x$ and $\vec y$. Let $\gamma$ be the angle between the lines $L$ and $M$, bisecting the angles formed by $\vec u$ and $\vec v$ and by $\vec x$ and $\vec y$ respectively. An easy computation shows that $\cos\beta=-1/5$ and $\cos(2\gamma)=5/7$. Therefore, by  \cite[Corollary~3.12]{Niven}, the angles $\alpha, \beta$ and $2\gamma$  are not a rational multiples of $\pi$.

Since $\vec u$ is conjugate to $\vec v$ and $\vec x$ is conjugate to $\vec y$ in $\R^2 \rtimes Q$, there must exist $c,d \in Q$ such that $c.\vec u=\vec v$ and $d.\vec x= \vec y$. 
Since $\alpha,\beta \notin \Q\pi$, using the classification of isometries of $\R^2$ we can conclude that the elements $c$ and $d$ must act on $\R^2$ as reflections in the lines $L$ and $M$ respectively. 
Then the product $cd \in Q$ must act as the rotation of $\R^2$ about the origin by $2\gamma$, which has infinite order as $2\gamma \notin \Q\pi$. This contradicts the assumption that $|Q|<\infty$, thus completing the proof.
\end{ex}

\begin{rem}\label{rem:Gardam}  Giles Gardam has  implemented a version of the algorithm described in \Cref{rem:algorithm}  in  \cite{GAP4}, checking property (VRC) 
for the tubular groups
\begin{equation}\label{eq:G_mnkl}
G_{m,n,k,l} \coloneq \langle a,b,s,t \mid [a,b]=1,~sa s^{-1} = b,~t a^mb^n t^{-1} =a^k b^l  \rangle,
\end{equation}
where  $(m,n),\, (k,l) \in \Z^2\setminus\{(0,0)\}$ (note that these groups generalize the family $G_{k,l}$ discussed in \Cref{ex:G_kl} and the group $G$ from \Cref{ex:3-dim}). Considering $m,n,k,l$ of absolute value at most $5$,  only looking at abelianizations of subgroups of index at most $6$ in  $G_{m,n,k,l}$, and  factoring out the obvious symmetries giving isomorphic groups, the algorithm identified $408$ different tuples $(m,n,k,l)$ for which $G_{m,n,k,l}$ has (VRC).
\end{rem}

    
\section{HNN-extensions of abelian groups}\label{sec:HNNs}
In this section we will give sufficient criteria for HNN-extensions of abelian groups to have (VRC). Here the situation is more complicated than in the case of amalgamated free products as the example of {Baumslag-Solitar groups} $BS(k,l)$ shows (see \eqref{eq:BS(k,l)}). Any group $BS(k,l)$ is an HNN-extension of the infinite cyclic group. It is known that this group is cyclic subgroup separable if and only if $l=\pm k$ (see \cite{Stebe,Kim_Tang_1999}). In particular, if $|k| \neq |l|$ then $BS(k,l)$ does not have (VRC), by \Cref{lem: props sep}.(ii).

\begin{defn}\label{def:balanced} A group $G$ is called \emph{balanced} provided the following condition holds for every infinite order element $g \in G$. If $g^k$ is conjugate to $g^l$ in $G$, for some  $k,l \in \Z\setminus\{0\}$, then $l=\pm k$.    
\end{defn}

In \cite[Corollary after Theorem~7]{Stebe} Stebe showed that every cyclic subgroup separable group $G$ is balanced.  Combining this fact with \Cref{lem: props sep}.(ii), we deduce that not being balanced is an obstruction to having property (VRC).

\begin{rem}\label{rem:VRC->balanced}
Any group with property (VRC)  is balanced.
\end{rem}

\begin{defn}\label{def:near_lin_indep} Let $\{\vec v_i\}_{i \in I}$ be a collection of vectors in $\R^n$. We will say that this collection is \emph{nearly linearly independent} if there exists a subset $J \subseteq I$ such that 

\begin{itemize}
    \item the vectors $\{\vec v_j\}_{j \in J}$ are linearly independent in $\R^n$;
    \item for each $i \in I$ there exists $j \in J$ and $\varepsilon \in \{\pm 1\}$ such that $\vec v_i=\varepsilon \, \vec v_j$.
\end{itemize}

If $A$ is a finitely generated abelian group, we will say that  elements $a_1, \dots,a_m \in A$ are \emph{nearly linearly independent} if the images of these elements in $A \otimes_\Z \R$ are nearly independent.
\end{defn}

In other words, a collection of vectors is nearly linearly independent if it becomes linearly independent after discarding all repetitions and one vector out of each mutually inverse pair of vectors in the collection.

\begin{prop}\label{prop:nearly_lin_idep_for_free_ab}
Let $A \cong \Z^n$ be a free abelian group, for some $n \in \N_0$,    and let $a_1,\dots,a_m$, $b_1,\dots,b_m$ be non-trivial elements in $A$. If these elements are nearly linearly independent then the multiple HNN-extension
\begin{equation}\label{eq:multiple_HNN}
G=\langle A,t_1,\dots,t_m \mid t_i a_i t_i^{-1}=b_i,~i=1,\dots,m \rangle   
\end{equation}
has property (VRC).
\end{prop}

\begin{proof} Let $\psi:A \to \R^n$ denote the group monomorphism arising from the $\Z$-module embedding $A \to A\otimes_\Z \R \cong \R^n$.
Let $\vec v_1, \dots, \vec v_m \in \R^n$ denote the images of $a_1,\dots, a_m$, and let $\vec v_{m+1}, \dots, \vec v_{2m} \in \R^n$ denote the images of $b_1,\dots,b_m$, respectively. By the assumptions, we can choose a subset $J \subseteq \{1,\dots,2m\}$ such that the vectors $\{\vec v_j\}_{j \in J}$ are linearly independent and for each $i \in \{1,\dots,2m\}$ there exist $j=j(i) \in J$ and $\varepsilon=\varepsilon(i) \in \{\pm 1\}$ such that $\vec v_i =\varepsilon\, \vec v_j$.

Complete the linearly independent set $\{\vec v_j\}_{j \in J}$ to a basis $\{\vec e_1,\dots ,\vec e_n\}$ of $\R^n$, and let $Q$ be the finite subgroup of $\mathrm{GL}(n,\R)$ arising as the group of signed permutations of this basis (see \Cref{rem:signed_perm}). We can now define a group homomorphism 
\[\varphi:G \to \R^n \rtimes Q\] as follows. The restriction of $\varphi$ to the base group $A$ is $\psi$ (we identify $\R^n$ with the subgroup $(\R^n,1) \leqslant \R^n \rtimes Q$), and for each $i \in \{1,\dots,m\}$ we let 
$\varphi(t_i) \in Q$ be the signed permutation interchanging $\varepsilon (i)\, \vec e_{j(i)}$ with $\varepsilon (i+m) \,\vec e_{j(i+m)}$ and fixing all remaining vectors $\vec e_k$, for $k \in \{1,\dots,n\} \setminus \{j(i),j(i+m)\}$.

By construction, the homomorphism $\varphi$ is  injective on $A$, therefore $G$ has (VRC) by \Cref{thm:hom_to_Rn_semidir_fin}.    
\end{proof}

We shall now extend \Cref{prop:nearly_lin_idep_for_free_ab} to multiple HNN-extensions with arbitrary finitely generated abelian bases.

\begin{thm}\label{thm:nearly_lin_indep_crit} Let $A$ be a group with a finite normal subgroup $H \n A$ such that $A/H \cong \Z^n$, for some $n \in \N_0$, and let $\xi:A \to \Z^n$ be the quotient map. Suppose that $a_1,\dots,a_m$, $b_1,\dots,b_m$ are infinite order elements in $A$ such that the $\xi$-images of these elements form a nearly linearly independent collection  in $\Z^n$. Then the multiple HNN-extension \eqref{eq:multiple_HNN}
has (VRC).    
\end{thm}

\begin{proof} Since the elements $\xi(a_i)$ and $\xi(b_i)$ have infinite order, for all $i=1,\dots,m$,
we can define the multiple HNN-extension $G_1$, of $\xi(A)=\Z^n$, as follows:
\[G_1=\langle \xi(A), s_1,\dots,s_m \mid s_i \xi(a_i) s_i^{-1}=\xi(b_i),~i=1,\dots,m \rangle.   \]
Clearly, we have a homomorphism $\varphi_1: G \to G_1$, whose restriction to $A$ is $\xi$ and such that $\varphi_1(t_i)=s_i$, for $i=1,\dots,m$.

Since $H \n A$ is a finite normal subgroup and $A/H \cong \Z^n$, $A$ is a finitely generated FC-group, hence its center has finite index. Therefore, there exists a finite index subgroup $K \n_f A$ such that $K \leqslant \cent(A)$ and $K \cap H=\{1\}$. In particular,  $A$ is finitely generated and virtually abelian. 

Denote $k=|A:K|\in \N$ and let 
$L \leqslant K$ be the subgroup generated by the elements $\{a_i^k,b_i^k \mid i=1,\dots,m\}$. Note that $L$ is central in $A$.
Since the collection 
$\{\xi(a_i),\xi(b_i)\}_{i=1}^m$ is nearly linearly independent in $\Z^n$, by the assumptions,  the images of the elements $\xi(a_i)$ and $\xi(b_i)$ in $\xi(A)/\xi(L)$ have order $k$, for all $i=1,\dots,m$. It follows that $\eta(a_i)$ and $\eta(b_i)$ have order $k$ in $A/L$, for all $i$, where $\eta:A \to A/L$ denotes the quotient map. Therefore, we have a homomorphism $\varphi_2:G \to G_2$, where $G_2$ is the multiple HNN-extension 
\[G_2=\langle A/L, r_1,\dots,r_m \mid r_i \eta(a_i) r_i^{-1}=\eta(b_i),~i=1,\dots,m \rangle.\] 
The restriction $\varphi_2$ to $A$ is $\eta$ and $\varphi_2(t_i)=r_i$, for $i=1,\dots,m$.

Note that the group $G_2$ has (VRC) as a multiple HNN-extension of the finitely generated virtually abelian group $A/L$ with finite associated subgroups, by \Cref{lem:VRC_preserved_by_free_contr_with_fin_edge_gps}. On the other hand, $G_1$ has (VRC) by \Cref{prop:nearly_lin_idep_for_free_ab}, whence $G_1 \times G_2$ has (VRC) by \Cref{lem: listresultsretr}.(viii). 

The homomorphism $\varphi \coloneq \varphi_1 \times \varphi_2:G \to G_1 \times G_2$ is injective on $A$, by construction. Indeed, since $\ker\varphi=\ker\varphi_1 \cap \ker\varphi_2$, we have
\[A \cap \ker\varphi=(A \cap \ker\varphi_1) \cap \ker\varphi_2=\ker \xi \cap \ker\varphi_2=H \cap L=\{1\}.\] Since $A$ is finitely generated and virtually abelian, \Cref{lem: props sep}.(ii) implies that $\varphi(A) \vr G_1 \times G_2$, hence $A \vr G$ by \Cref{lem: listresultsretr}.(i). We can now apply \Cref{cor:homom_to_vab->VRC} to conclude that $G$ has (VRC), as required.
\end{proof}

\begin{rem}\label{rem:thm_applies_to_all_ab_gps} Every finitely generated abelian group $A$ splits as  a direct product $H \times K$, where $H$ is a finite group and $K \cong \Z^n$, for some $n \in \N_0$. Thus \Cref{thm:nearly_lin_indep_crit} can be applied to $A$. Moreover, elements $a_1,\dots,a_k \in A$ are nearly linearly independent if and only if their images in $K$ (under the natural projection $A \to K$) are nearly linearly independent.    
\end{rem}

\begin{lemma}\label{lem:balanced_HNN} Let $A$ be a group with two elements $a,b \in A$ of the same order, and let $G$ be the HNN-extension
\begin{equation}\label{eq:single_HNN_of_A}
			G = \langle A,t \mid  t a t^{-1} = b \rangle.
		\end{equation}
If $G$ is balanced then either $\langle a \rangle \cap \langle b \rangle=\{1\}$ in $A$ or there exists $m \in \N$ such that $b^m=a^{\pm m}$.
\end{lemma}

\begin{proof}
Suppose that $a^k=b^l$, for some $k,l \in \Z \setminus \{0\}$, then 
\[ta^{l}t^{-1}=b^l=a^k~\text{ in } G.\]
Since $G$ is balanced,  we must either have $|k|=|l|$ (and then we set $m\coloneq |k| \in \N$) or $a$ and $b$ must have the same finite order  in $A$ (and then we let $m \in \N$ be this order). In either case, we will have $b^m=a^{\pm m}$.    
\end{proof}

\begin{cor}\label{cor:balanced HNN VRC}
Let $G$ be the HNN-extension \eqref{eq:single_HNN_of_A}, where $A$ is a finitely generated abelian group. 	
Then the following are equivalent:
\begin{itemize}
    \item[(i)] $G$ has property (VRC);
    \item[(ii)] $G$ is balanced;
    \item[(iii)] either $\langle a \rangle \cap \langle b \rangle=\{1\}$ or  $b^m=a^{\pm m}$ in $A$, for some $m \in \N$.
\end{itemize}    
\end{cor}

\begin{proof} Statement (i) implies (ii) by \Cref{rem:VRC->balanced}, and (ii) implies (iii) by \Cref{lem:balanced_HNN}.

It remains to show that (iii) implies (i). If $a$ and $b$ have finite order then $G$ has (VRC) by \Cref{lem: listresultsretr}.(iv) and \Cref{lem:VRC_preserved_by_free_contr_with_fin_edge_gps}. Thus we can suppose that the order of $a$ and $b$ is infinite.

If $\langle a \rangle \cap \langle b \rangle=\{1\}$ in $A$ then the elements $a$ and $b$ are linearly independent in $A$, so $G$ has (VRC) by \Cref{thm:nearly_lin_indep_crit}. Therefore, we can assume that $b^m=a^{\varepsilon m}$, for some $m \in \N$ and $\varepsilon \in \{\pm 1\}$. Then the images $\vec u$ and $\vec v$, of $a$ and $b$ in $A \otimes_{\Z} \R$, are non-zero and satisfy $\Vec v= \varepsilon \, \vec u$. Thus $a$ and $b$ are nearly linearly independent in $A$, so \Cref{thm:nearly_lin_indep_crit} again applies to show that $G$ has (VRC), as required.    
\end{proof}

\begin{ex} \Cref{cor:balanced HNN VRC} can be compared to \Cref{cor: BS has (LR)}, characterizing property (LR) in the case of  Baumslag-Solitar groups. Note that, in general, an HNN-extension of $\Z^2$ may have (VRC) but not have (LR). Indeed, the \emph{Burns-Karrass-Solitar group}
\[G=\langle a,b,t \mid [a,b]=1,~tat^{-1}=b \rangle\]
has (VRC) by  \Cref{cor:balanced HNN VRC}, but it is not LERF by \cite{BKS}. Thus $G$ does not have (LR), by \Cref{lem: props sep}.(iii).
\end{ex}

\begin{ex}\label{ex:single_HNN_of_vab-not_VRC} Let us show that \Cref{cor:balanced HNN VRC} does not have a straightforward generalization to balanced HNN-extensions of virtually abelian groups.

Choose any $k \in \Z$, with $|k| \ge 2$, and let $H_k$ be the group with presentation \eqref{eq:pres_of_H_k}. Then $H_k$ is an HNN-extension of $A \cong \Z \wr C_2$ with associated cyclic subgroups $\langle b \rangle$ and $\langle ab^k \rangle$, and $H_k$ does not have (VRC), as shown in \Cref{ex:H_k}. 

However, it is easy to see that $A$ is \emph{quasi-regular at the pair $\{b,ab^k\}$} (see \cite[Definition~2.4]{Kim_Tang_1999}) because these two elements are primitive (in fact, they form a basis) in the free abelian group $B \coloneq \langle a,b \rangle \leqslant_f A$, and so they have the same order $m$ in the quotient $A/B^m$, for each $m \in \N$. Therefore, we can apply a theorem of Kim and Tang \cite[Theorem~2.9]{Kim_Tang_1999} to deduce that $H_k$ is cyclic subgroup separable (and so it is balanced by \cite[Corollary after Theorem~7]{Stebe}), for every $k \in \Z$.     
\end{ex}
	

\section{General graphs of abelian groups} \label{sec:gen_graphs_of_gps}
The purpose of this section is to extend the ``near linear independence'' criteria discussed in \Cref{sec:HNNs} from multiple HNN-extensions  to more general graphs of finitely generated abelian groups.

Let $(\mathcal{G},\Gamma)$ be a finite graph of groups. 
Suppose that there is a  maximal tree $T$ in $\Gamma$ such that for every $e \in E\Gamma \setminus ET$ the edge group $G_e=\langle c_e \rangle$ is cyclic, and denote $a_e \coloneq \alpha_e(c_e) \in G_{\alpha(e)}$.

We can consider the subgraph of groups $(\mathcal{G}_T,T)$, determined by $T$, where  
\begin{equation*}\label{eq:def_of_G_T}
\mathcal{G}_T=\Bigl(\{G_v\}_{v\in V\Gamma},\{G_e\}_{e \in ET} ,\{\alpha_e\}_{e \in ET}\Bigr)    
\end{equation*}
(remember that $V\Gamma=VT$).
Let $E\Gamma=E\Gamma^+ \sqcup E\Gamma^-$ be an orientation on $E\Gamma$. 
Choose the induced orientation $ET=ET^+ \sqcup ET^-$ on the edges of $T$ (i.e., $ET^+=ET \cap E\Gamma^+$), and consider the fundamental group 
$G_T \coloneq \pi_1(\mathcal{G}_T,T,T, ET^+)$. Evidently, we can re-write presentation \eqref{eq:pres_of_fund_gp}, of the fundamental group $G \coloneq \pi_1(\mathcal{G},\Gamma,T, E\Gamma^+)$,  as
\begin{equation}\label{eq:new_pres_of_fund_gp}
G=\langle G_T,~\{t_e\}_{e \in E\Gamma^+ \setminus ET} 
\mid t_e\, a_{\ov{e}} \, t_e^{-1}=a_e,
\text{ for all }e \in E\Gamma^+ \setminus ET\rangle.    
\end{equation}

In other words, $G$ is a multiple HNN-extension of $G_T$ with free letters $t_e$, $e \in E\Gamma^+ \setminus ET$, and with cyclic associated subgroups $\{\langle a_e \rangle\}_{e \in E\Gamma \setminus ET}. $

\begin{thm}\label{thm:gen_graphs_of_ab_gps} Using the notation above, suppose that all vertex group $G_v$, $v \in V\Gamma$, are finitely generated and abelian. Let $A$ be the abelianization of $G_T$ and let $\rho:G_T \to A$ be the abelianization homomorphism.

If the collection  $\{\rho(a_e)\}_{e \in E\Gamma \setminus ET}$ is nearly linearly independent in $A$ then $G=\pi_1(\mathcal{G},\Gamma,T, E\Gamma^+)$ has property (VRC).
\end{thm}

\begin{proof}
By \Cref{lem: graph of groups - tree - hom to v.a}, the 
homomorphism $\rho$ is injective on each subgroup $\alpha_e(G_e)$, $e \in E\Gamma \setminus ET$, of the multiple HNN-extension \eqref{eq:new_pres_of_fund_gp}, so we have a homomorphism $\varphi:G \to G_1$, where $G_1$ is a multiple HNN-extension of the finitely generated abelian group $A$:
\[G_1=\langle A,~\{s_e\}_{e \in E\Gamma^+ \setminus ET} 
\mid s_e\, \rho(a_{\ov{e}}) \, s_e^{-1}=\rho(a_e),
\text{ for all }e \in E\Gamma^+ \setminus ET\rangle,\]
and the restriction of $\varphi$ to $G_T$ is $\rho$
(this is a special case of the morphism discussed in \Cref{ex:morphism_between_fund_gps-1}).

Since the collection $\{\rho(a_e)\}_{e \in E\Gamma \setminus ET}$ is nearly linearly independent in $A$, \Cref{thm:nearly_lin_indep_crit} implies that the group $G_1$ has (VRC).
Given any $v \in V\Gamma$, $\varphi(G_v)=\rho(G_v)$ is a finitely generated abelian subgroup of $G_1$, so $\varphi(G_v) \vr G_1$,   by 
\Cref{lem: props sep}.(ii). We also know, from \Cref{lem: graph of groups - tree - hom to v.a}, that $\varphi$ is injective on $G_v \leqslant G_T$, hence we can apply \Cref{lem: listresultsretr}.(i) to deduce that $G_v \vr G$, for all $v \in V\Gamma$. Therefore, $G$ has (VRC) by \Cref{cor:homom_to_vab->VRC}.
\end{proof}

\begin{ex}\label{ex: graph of abelian n.l.i}
    Let $B$ and $C$ be copies of $\Z^3$ and $\Z^2$ respectively. Fix some free abelian bases $\{ b_1,b_2,b_3 \}$ of $B$, and $\{c_1,c_2\}$ of $C$. Let $G$ be the group with (relative) presentation
    \begin{multline*}
        G = \langle B,C,t_2,t_3 \mid b_1^3 = c_1^2, ~b_2^4 = c_2^3,~ t_2(c_1^2c_2^2)t_2^{-1} = b_1^2b_2^2,\\
        t_3(b_1b_2b_3)t_3^{-1} = b_1^{-2}b_2^{-2}\rangle.
    \end{multline*}
Thus $G$ is the fundamental group of the graph of groups $(\mathcal{G},\Gamma)$ depicted on \Cref{fig:ex_graph_of_groups} (by convention, we only draw the positively oriented edges), where we choose the maximal tree $T$ to consist of the two vertices and the edges $e_1$ and $\ov{e}_1$. The vertex groups in this graph of groups are $B$ and $C$, the edge groups are $G_{e_1}=\Z^2$ and $G_{e_2}=G_{e_3}=\Z$. We also have $\alpha_{e_1}(G_{e_1})=\langle b_1^3,b_2^4 \rangle$, $\alpha_{\ov{e}_1}(G_{\ov{e}_1})=\langle c_1^2,c_2^3 \rangle$ and $\alpha_{{e}}(G_{{e}})=\langle a_{e} \rangle$, for $e \in E\Gamma \setminus \{e_1,\ov{e}_1\}$, where $a_{e_2} = b_1^2b_2^2$,  $a_{\ov{e}_2} = c_1^2c_2^2$, $a_{e_3} = b_1^{-2}b_2^{-2}$ and $a_{\ov{e}_3} = b_1b_2b_3$.
   \begin{figure}[H]
            \centering
            \begin{tikzpicture}
            
                \draw[fill=black] (-2,0) circle (2pt) node[](1){} node[left=3pt](B){$B$}; 

    		\draw[fill=black] (2,0) circle (2pt) node[](2){} node(B)[right=2pt]{$C$};

    		\path [thick,draw=black,postaction={on each segment={mid arrowa={black,scale=1.2}{0.53}}}]
                
                (1.center) -- node[below=-1pt]{\small $e_1$} node[above = 0.1pt]{\small$\Z^2$}(2.center) ;
                \path [thick,draw=black,postaction={on each segment={mid arrowa={black,scale=1.2}{0.53}}}]
                (1.center) to[bend right = 60] node[above]{$\Z$} node[below=-1pt]{\small $e_2$} (2.center);
                \path [thick,draw=black,postaction={on each segment={mid arrowa={black,scale=1.2}{0.52}}}]
                (1.center) to[loop, min distance = 25mm] node[above]{$\Z$} node[below=-1pt]{\small $e_3$} (1.center);
    	\end{tikzpicture}\caption{The graph of groups $(\mathcal{G},\Gamma)$ in \Cref{ex: graph of abelian n.l.i} } \label{fig:ex_graph_of_groups}
        \end{figure}
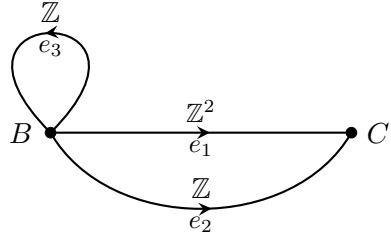

We will now verify that $G$ has (VRC) by using the criterion from \Cref{thm:gen_graphs_of_ab_gps}.
Indeed, let $G_T$ be the fundamental group of the subgraph of groups determined by $T$, and let $\rho\colon G_T \rightarrow A$ be the abelianization homomorphism. According to \Cref{thm:gen_graphs_of_ab_gps}, we only need to check that the collection $\{\rho(c_1^2c_2^2), \rho(b_1^2 b_2^2),\rho(b_1b_2b_3), \rho(b_1^{-2} b_2^{-2})\}$ is nearly linearly independent in $A$. Note that
\[A \cong(B \times C)/\langle (b_1^3, c_1^{-2}),  (b_2^4, c_2^{-3}) \rangle,\]
so it is sufficient to work in the free abelian group $\tilde  A \coloneq B \times C \cong \Z^5$, with basis $\{b_1,b_2,b_3,c_1,c_2\}$, and check that the extended collection 
\[\{c_1^2c_2^2, b_1^2 b_2^2, b_1b_2b_3, b_1^{-2} b_2^{-2}, b_1^3 c_1^{-2},  b_2^4 c_2^{-3}\} \] is nearly linearly independent in $\tilde{A}$. Note that the fourth element in the collection is the inverse of the second one, so we can discard it and write the remaining elements as columns in the $5 \times 5$ matrix
\[M \coloneq \begin{pmatrix}
        0 & 2 & 1 & 3 & 0\\ 
        0 & 2 & 1 & 0 & 4  \\ 
        0 & 0 & 1 & 0 & 0\\    
        2 & 0 & 0 & -2& 0 \\ 
        2 & 0 & 0 & 0 & -3  
        \end{pmatrix}.
\]
One easily checks that $\mathrm{rank}(M)=5$, so its column vectors are linearly independent in $\tilde A \otimes_{\Z} \R$. It follows that the first three column vectors are linearly independent in $A\otimes_{\Z} \R$, which is the quotient of $\tilde A\otimes_{\Z} \R$ by the subspace generated by last two vectors. Hence, $G$ has (VRC) by \Cref{thm:gen_graphs_of_ab_gps}.

\end{ex}

A similar idea can be used to extend \Cref{cor:balanced HNN VRC} to a more general situation. Recall that we define graphs in Subsection~\ref{subsec:graphs_of_gps} using Serre's approach, so  the \emph{Euler characteristic} of a finite graph $\Gamma$  is \[\chi(\Gamma) \coloneq |V\Gamma|-\frac12 |E\Gamma|.\] For a connected finite graph $\Gamma$, one always has $\chi(\Gamma) \le 1$, with equality if and only if $\Gamma$ is a tree  (see \cite[Proposition~12 in Section~I.2.3]{Serre}).

\begin{cor}\label{cor:one_edge_outside_T} Let $G$ be the fundamental group of a  graph of groups $(\mathcal{G},\Gamma)$, where 
$\Gamma$ is a finite connected graph with $\chi(\Gamma) = 0$. Suppose that all vertex groups $G_v$, $v \in V\Gamma$, are finitely generated abelian  and there exists an edge $e \in E\Gamma$ such that $e \notin ET$, for some maximal tree $T$ in $\Gamma$, and $G_e$ is cyclic. Then $G$ has property (VRC) if and only if $G$ is balanced.
\end{cor}

\begin{proof} The necessity is given by \Cref{rem:VRC->balanced}, so it remains to show that if $G$ is balanced then it has (VRC). 

Choose a maximal tree $T$ and an edge $e$ in $\Gamma$, as given in the statement. Since $\chi(\Gamma)=0$, we have $E\Gamma=ET \sqcup \{e, \ov{e}\}$. Fix an orientation $E\Gamma=E\Gamma^+ \sqcup E\Gamma^-$ such that $e \in E\Gamma^+$. 
Let $ G_T\coloneq \pi_1(\mathcal{G}_T,T,T, ET^+)$, as defined in the beginning of the section. By the assumptions, $G_e=\langle c \rangle$, and we set $\ot{a} \coloneq \alpha_{\ov{e}}(c) \in G_{\alpha(\ov{e})} \leqslant G_T$ and $\ot{b} \coloneq \alpha_e(c) \in G_{\alpha(e)} \leqslant G_T$, so that
\begin{equation}\label{eq:single_HNN_ext_of_G_T}
    G \cong \langle G_T, t \mid t\,\ot{a}\,t^{-1}=\ot{b} \rangle.
\end{equation} 
Arguing as in the proof of \Cref{thm:gen_graphs_of_ab_gps}, we obtain a homomorphism $\varphi: G \to G_1$, where $G_1$ is the HNN-extension \eqref{eq:single_HNN_of_A}, where $A$ is the abelianization  of $G_T$
and $a,b \in A$ are the $\varphi$-images of the elements $\ot{a},\ot{b} \in G_T$, respectively.

Since $G$ is balanced, in view of \eqref{eq:single_HNN_ext_of_G_T} and \Cref{lem:balanced_HNN}, we know that either $\langle \ot{a} \rangle \cap \langle\ot{b}\rangle=\{1\}$ or $\ot{b}^m=\ot{a}^{\pm m}$ in $G_T$, for some $m \in \N$. And the same holds for the elements $a$ and $b$ in $A$, because $\varphi$ is injective on the union of the vertex groups $G_{\alpha(\ov{e})} \cup G_{\alpha(e)}$ in $G$ (by \Cref{lem: graph of groups - tree - hom to v.a}). Therefore, $G_1$ has (VRC) by \Cref{cor:balanced HNN VRC}. The argument from the last paragraph of the proof of \Cref{thm:gen_graphs_of_ab_gps} now shows that $G$ has (VRC).
\end{proof}

\Cref{cor:balanced_circuits->VRC} from the Introduction follows by combining \Cref{thm:tree of abelian->VRC} (when $\chi(\Gamma)=1$) with \Cref{cor:one_edge_outside_T} (when  $\chi(\Gamma)=0$).

 \begin{rem}\label{rem:determining_balanced} 
It is easy to determine whether the group $G$ from \Cref{cor:one_edge_outside_T} is balanced. As the proof of that corollary shows, 
this is true if and only if either $\langle {a} \rangle \cap \langle {b}\rangle=\{1\}$ or ${b}^m={a}^{\pm m}$ in $A$, for some $m \in \N$, where $A$ is the abelianization of $G_T$ and $a,b$ are generators of the images of $\omega_e(G_e)$ and $\alpha_e(G_e)$ in $A$, respectively.  
\end{rem}


\section{(VRC) implies CAT($0$)} \label{sec:CAT(0)}
Recall that a group $G$ acts on a metric space $(X,\mathrm{d})$ \emph{properly} if for  for each $x \in X$ there exists $r>0$ such that the set $\{g \in G \mid \mathrm{d}(g.x,x) \le r\}$ is finite (see \cite[Definition~I.8.2]{Bridson_Haefliger}). Obviously, a discrete subgroup of $\R^n$, equipped with a Euclidean metric, acts on it properly by translations.

A \emph{CAT($0$)} space is a geodesic metric space in which all geodesic triangles satisfy the CAT($0$) inequality, see \cite[Section~II.1.1]{Bridson_Haefliger}. A group is said to be \emph{CAT($0$)} if it admits a proper and cocompact action on a complete CAT($0$) space. 
In this section we will show that fundamental groups of finite graphs of finitely generated virtually abelian groups with property (VRC) are CAT($0$). To this end we will use the equivariant gluing theorems for amalgamated products and HNN-extensions of CAT($0$) groups from the book of Bridson and Haefliger \cite[Chapter~II.11]{Bridson_Haefliger}. We start by extending these theorems to general graphs of groups.

\begin{lemma}\label{lem:cat(0)_for_gen_graphs_of_groups}
Let $(\mathcal{G},\Gamma)$ be a finite graph of groups, let $T$ be a maximal tree in $\Gamma$, let $E\Gamma=E\Gamma^+ \sqcup E\Gamma^-$ be an orientation on $E\Gamma$ and let $G=\pi_1(\mathcal{G},\Gamma,T,E\Gamma^+)$. Suppose that each vertex group $G_v$ admits a proper cocompact action by isometries on a complete CAT($0$) space $X_v$, and for each $e \in E\Gamma^+$ the edge group $G_e$ acts properly and cocompactly by isometries on a complete CAT($0$) space $Y_e$. 

Assume that for every $e \in E\Gamma^+$ there exist $\alpha_e$- and $\omega_e$-equivariant isometric embeddings $\mu_e: Y_e \to X_{\alpha(e)}$ and $\varkappa_e:Y_e \to X_{\omega(e)}$, that is
\[\mu_e(g.y)=\alpha_e(g).\mu_e(y) \text{ and } \varkappa_e(g.y)=\omega_e(g).\varkappa_e(y), \text{ for all } g \in G_e,~y \in Y_e.\]
Then $G$ acts properly and cocompactly on a complete CAT($0$) space $W$, and for every $v \in V\Gamma$ there is a $G_v$-equivariant isometric embedding $X_v \to W$.
\end{lemma}

\begin{proof}
We will construct the desired CAT($0$) space $W$ by induction on the number of edges $k \coloneq |E\Gamma|$. If $k=0$, i.e., $\Gamma$ consists of a single vertex $v$ then there is nothing to prove, as $G=G_v$ and we can take $W=X_v$. 

Assume that $k \ge 1$ and the statement has already been established for all graphs with $k-1$ edges. Suppose, first, that $\Gamma$ is a tree, i.e., $\Gamma=T$. Choose a leaf vertex $w \in VT$, so that there is a single edge $e \in E\Gamma$ such that $w=\alpha(e)$. Without loss of generality, we can assume that $e \in E\Gamma^+$. Let $\Gamma_1$ be the tree obtained from $\Gamma$ by removing $w$ together with the edges $e, \ov{e}$, and let $\mathcal{G}_1$ be the restriction of $\mathcal{G}$ to $\Gamma_1$. Set $E\Gamma_1^+ \coloneq E\Gamma_1 \cap E\Gamma^+$ and  $H \coloneq \pi_1(\mathcal{G}_1,\Gamma_1,\Gamma_1,E\Gamma_1^+)$. 

Then $G \cong G_w*_{G_e} H$ can be obtained by taking the free product of $G_{w}$ with $H$, amalgamating the subgroups $\alpha_{e}(G_e) \leqslant G_w$ and $\omega_{e}(G_{e}) \leqslant H$. By the induction hypothesis, $H$ acts properly and cocompactly by isometries on a complete CAT($0$) space $W_1$, and for each $v \in V\Gamma_1$ there is a  $G_v$-equivariant isometric embedding $\nu_v: X_v \to W_1$. Therefore, $\mu_e:Y_e \to X_w$ and $\nu_{\omega(e)} \circ\varkappa_e :Y_e \to W_1$ are  $\alpha_e$- and $\omega_e$-equivariant isometric embeddings, so we can apply
\cite[Proposition~II.11.18]{Bridson_Haefliger} to obtain a complete CAT($0$) space $W$ with a proper cocompact action of $G$. Moreover, the proof of this statement in \cite{Bridson_Haefliger} shows that there is 
a $G_w$-equivariant isometric embedding $X_w \to W$ and
an $H$-equivariant isometric embedding $W_1 \to W$. By composing the latter maps with $\nu_v$, we get $G_v$-equivariant isometric embeddings $X_v \to W$, for all $v \in V\Gamma$.

Thus we can now suppose that $\Gamma$ is not a tree, so 
there is an edge $e \in E\Gamma^+ \setminus ET$. Let $\Gamma_1$ be the connected subgraph of $\Gamma$ obtained by removing this edge and its inverse, so that $V\Gamma_1=V\Gamma$ and $E\Gamma_1=E\Gamma \setminus\{e,\ov{e}\}$. Let $\mathcal{G}_1$ be the restriction of $\mathcal{G}$ to $\Gamma_1$, and let  $E\Gamma^+_1 \coloneq E\Gamma_1 \cap E\Gamma^+$. Finally, we define $H \coloneq \pi_1(\mathcal{G}_1,\Gamma_1,T,E\Gamma_1^+)$.

By construction, $G$ is isomorphic to an HNN-extension of $H$ with associated subgroups $\alpha_{e}(G_{e})$ and $\omega_{e}(G_{e}) $. By induction, $H$ acts properly and cocompactly on a complete CAT($0$) space $W_1$, and there are $G_v$-equivariant isometric embeddings $\nu_v:X_v \to W_1$, for all $v \in V\Gamma_1$. If we denote $w \coloneq \alpha(e)$ and $u \coloneq \omega(e)$ in $V\Gamma$, then $\nu_w\circ \mu_e$ and $\nu_u\circ \varkappa_e$ are $\alpha_e$- and $\omega_e$-equivariant isometric embeddings $Y_e \to W_1$. Therefore, we can apply \cite[Proposition~II.11.21]{Bridson_Haefliger} to conclude that $G$ admits a proper and cocompact action by isometries on a complete CAT($0$) space $W$, with an $H$-equivariant isometric embedding $W_1 \to W$. By composing this embedding with the map $\nu_v$, we obtain a $G_v$-equivariant isometric embedding $X_v \to W$, for each $v \in V\Gamma$. This completes the step of induction, finishing the proof of the lemma.    
\end{proof}

\begin{lemma}\label{lem:G-invar_affine_subspace}
Suppose that  $n \in \N_0$ and $G$ is a
finitely generated group acting on $\R^n$ properly by affine Euclidean isometries such that a finite index subgroup of $G$ acts by translations. Then there exists an affine subspace $X$ of $\R^n$ which is invariant under the action of $G$ and $G$ acts on $X$ properly and cocompactly.

Moreover, if $Z$ is any $G$-invariant affine subspace of $\R^n$ then there exists a $G$-equivariant isometric embedding $X \to Z$.
\end{lemma}

\begin{proof}
Since every isometry of $\R^n$ is semi-simple (\cite[Proposition~II.6.5]{Bridson_Haefliger}),
the existence of the affine subspace $X$ with the required properties is a consequence of the Flat Torus Theorem: see \cite[Corollary II.7.2]{Bridson_Haefliger}.

Now, suppose that $Z$ is an affine subspace of $\R^n$ invariant under the action of $G$. By the first claim, there exists an affine subspace
$X'$ of $Z$ such that $G$ preserves $X'$ and acts on it cocompactly. Let $A \leqslant_f G$ be a finite index subgroup acting on $\R^n$ by translations. Thus, for every $a \in A$ there is a vector $\vec u_a \in \R^n$ such that $a.y=y+\vec u_a$, for all $y \in \R^n$. 

Choose any point $x \in X$. Since $|G:A|<\infty$, $A$ acts on $X$ cocompactly, therefore there is $r \ge 0$ such that every point of $X$ is at distance at most $r$ from the $A$-orbit of $x$. It follows that $X=x+V$, where $V$ the linear subspace of $\R^n$ spanned by $\{\vec u_a \mid a \in A\}$. Similarly, $X'=x'+V$, for any point $x' \in X'$.

In particular, $X'=X+x'-x$, so these two affine subspaces are parallel and have the same dimension. It follows that there is a unique vector $\vec w \in \R^n$ such that $\vec w$ is orthogonal to both $X$ and $X'$ and $x+\vec w \in X'$ (the norm $\|\vec w\|$ is the Euclidean distance between the parallel affine subspaces $X$ and $X'$).
Then $X + \vec w=X'$, and for every $y \in X$, $y+\vec w$ is the unique point of $X'$ at distance $\|\vec w\|$ from $y$.
Note that for all $g \in G$ and $y \in X$ we have \[\|g.(y+ \vec w)-g.y\|=\|y+\vec w-y\|=\|\vec w\|,\] because $G$ acts on $\R^n$ by isometries. So the distance from $g.y \in X$ to $g.(y+\vec w) \in X'$ is $\|\vec w\|$, hence, 
\begin{equation}\label{eq:gx''}
g.(y+\vec w)=g.y+ \vec w, ~\text{ for all } g \in G.
\end{equation}
Equation \eqref{eq:gx''} shows that the translation $y \mapsto y+\vec w$ defines an isometric embedding of $X$ into $Z$ which is equivariant with respect to the actions of $G$ on these affine subspaces.    
\end{proof}

We are now ready to prove the main result of this section.

\begin{proof}[Proof of \Cref{prop:VRC->CAT(0)}] 
Suppose that $(\mathcal{G},\Gamma)$ is a finite graphs of groups with finitely generated virtually abelian vertex groups.
Choose a maximal tree $T$ in $\Gamma$ and an orientation $E\Gamma=E\Gamma^+ \sqcup E\Gamma^-$, so that $G=\pi_1(\mathcal{G},\Gamma,T,E\Gamma^+)$ has presentation \eqref{eq:pres_of_fund_gp}.
 
Assume that $G$ has (VRC). By \Cref{thm:hom_to_Rn_semidir_fin}, there is a Euclidean-by-finite group $P=\R^n \rtimes Q$ and a homomorphism  $\varphi:G \to P$ such that $\varphi$ is injective on the vertex group $G_v$ and $\varphi(G_v)$ is a discrete subgroup of $P$, for each $v \in V\Gamma$. Consider the natural action of $P$ by isometries on the Euclidean space $\R^n$, as described in \Cref{rem:E-by-f_acts_by_isoms}. This induces 
an action of $G$ on $\R^n$ by affine Euclidean isometries, so that the restriction of this action to each vertex group $G_v$ is proper and the finite index subgroup $G_v \cap \varphi^{-1}((\R^n,1)) \leqslant_f G_v$ acts by translations.

According to \Cref{lem:G-invar_affine_subspace}, for every $v \in V\Gamma$, there is an affine subspace $X_v$ of $\R^n$ such that $G_v$ preserves $X_v$ setwise and acts on it properly and cocompactly. 
Similarly, for each $e \in E\Gamma^+$ there is an affine subspace $Y_e$ of $\R^n$ on which $\omega_e(G_e)$ acts properly and cocompactly. Since $\alpha_e(G_e)=t_e \omega_e(G_e)t_e^{-1}$ in $G$, $\alpha_e(G_e)$ preserves the affine subspace $Z_e \coloneq t_e.Y_e$ and acts on it  properly and cocompactly.

Since $\omega_e(G_e) \leqslant G_{\omega(e)}$, the affine subspace $X_{\omega(e)}$ is $\omega_e(G_e)$-invariant, so we can apply  \Cref{lem:G-invar_affine_subspace} to find isometric embeddings \[\varkappa_{e}: Y_e \to X_{\omega(e)}, ~e \in E\Gamma^+,\]
that are equivariant with respect to the actions of $\omega_e(G_e) \leqslant G$ on $\R^n$. For each $e \in E\Gamma^+$, define an action of $G_e$ on $Y_e$ by $g.y \coloneq \omega_e(g).y$, for all $y \in Y_e$ and $g\in G_e$. The $\omega_e(G_e)$-equivariance of $\varkappa_e$ yields
\begin{equation}\label{eq:kappa_e}
  \varkappa_{e}(g.y)=\omega_e(g).\varkappa_{e}(y),\text{ for all } e \in E\Gamma^+,~ g \in G_e \text{ and } y\in Y_e.  
\end{equation}
Similarly,  there are isometric embeddings 
$\lambda_e: Z_e \to X_{\alpha(e)}$, for all $e \in~E\Gamma^+$, such that 
\begin{equation}\label{eq:lambda_e}
\lambda_e(\alpha_e(g).z)=\alpha_e(g).\lambda_e(z),\text{ for all } g \in G_e \text{ and } z\in Z_e.    
\end{equation}
For each $e \in E\Gamma^+$ we define an isometric embedding $\mu_e:Y_e \to X_{\alpha(e)}$ by $\mu_e(y) \coloneq \lambda_e(t_e.y)$, for all $y \in Y_e$. Given any $g \in G_e$, we know that  $t_e\omega_e(g)=\alpha_e(g)t_e$ in $G$, so  for each $y \in Y_e$ we obtain
\[\mu_e(g.y)=\lambda_e(t_e.(g.y))=\lambda_e(t_e.(\omega_e(g).y))=\lambda_e(\alpha_e(g).(t_e.y)).\]
Equation \eqref{eq:lambda_e} now implies that 
\begin{equation}\label{eq:mu_e}
 \mu_e(g.y)=\alpha_e(g).\lambda_e(t_e.y)=\alpha_e(g).\mu_e(y), 
\end{equation}
 for all  $e \in E\Gamma^+$, $g \in G_e$  and  $y\in Y_e$.

Equations \eqref{eq:mu_e} and \eqref{eq:kappa_e} allow us to apply \Cref{lem:cat(0)_for_gen_graphs_of_groups} and deduce that $G$ acts properly cocompactly on a complete CAT($0$) space, as required.
\end{proof}


\section{(VRC) tubular groups are virtually free-by-cyclic}
The purpose of this section is give another application of the property (VRC) by observing that every tubular group with (VRC) is virtually (free of finite rank)-by-(infinite cyclic). 
Let us start with the following observation.

\begin{rem} Let $G$ be the fundamental group of a finite graph of groups $(\mathcal{G},\Gamma)$ where all vertex groups are finitely generated and virtually abelian. If $G$ has (VRC) then we claim it fits into a short exact sequence 
\[\{1\} \to N \to G \to P \to \{1\},\]
where $N \n G$ is free and $P \cong G/N$ is finitely generated and virtually abelian.
\end{rem}

Indeed, \Cref{prop:further_props_of_VRC} gives an epimorphism $\varphi:G \to P$, where $P$ is finitely generated and virtually abelian, such that $\ker\varphi \cap G_v=\{1\}$, for every vertex group $G_v$, $v \in V\Gamma$. It follows that $N \coloneq \ker\varphi$ acts freely on the \emph{Bass-Serre tree} corresponding to the given splitting of $G$, hence $N$ is a free group (see \cite[Theorem~4 in Section~I.3.3]{Serre}).

Thus property (VRC) has strong implications about the structure of $G$. In some cases, we can replace $G$ by a finite index subgroup to make sure that the free subgroup $N \lhd G$ is finitely generated. We will now explain this for tubular groups.

If $H$ is a fundamental group of a graph of groups $(\mathcal{H},\Delta)$ then $H$ acts on the corresponding {Bass-Serre tree} $\mathcal{T}$, so that vertex stabilizers are conjugates of the vertex groups $H_v$, $v \in V\Delta$, and edge stabilizers are conjugates of the edge groups $H_e$, $e \in E\Delta$. For any subgroup $G \leqslant H$, this induces an action of $G$ on $\mathcal{T}$, which, by the Structure Theorem in Bass-Serre Theory \cite[Section~I.5.4]{Serre}, gives rise to a splitting of $G$ as the fundamental group of a graph of groups $(\mathcal{G},\Gamma)$, where every vertex (edge) group is the $G$-stabilizer of some vertex (respectively, edge) of $\mathcal{T}$. Moreover, if $\Delta$ is finite and $|H:G|<\infty$ then $\Gamma$ is also finite.

\begin{lemma}\label{lem:VRC->homom_to_Z} Let $H$ be the fundamental group of a finite graph of groups $(\mathcal{H},\Delta)$ with virtually cyclic edge groups. Suppose that the image in $H$ of every edge group is a virtual retract of $H$. Then there exists a finite index normal subgroup $G \n_f H$ such that the following holds.

The subgroup $G$ splits as the fundamental group of a finite graph of groups $(\mathcal{G},\Gamma)$, as described above,  and there exists a homomorphism $\xi:G \to \Z$ whose restriction to every edge group $\alpha_e(G_e)$, $e \in E\Gamma$, is injective.    
\end{lemma}

\begin{proof} By \Cref{prop:further_props_of_VRC} and \Cref{rem:just_vr}, there exists a homomorphism $\psi:H \to P$, where $P$ is a finitely generated virtually abelian group, such that $\psi$ is injective on (the image of) each edge group in $H$. Let $A \n_f P$ be a finitely generated free abelian normal subgroup of finite index and set $G \coloneq \psi^{-1}(A) \n_f H$.

As discussed in the paragraph above the statement of the lemma, $G$ decomposes as the fundamental group of a finite graph of groups $(\mathcal{G},\Gamma)$, where each vertex (edge) group is the intersection of $G$ with a conjugate of a vertex (respectively, edge) group of $(\mathcal{H},\Delta)$ in $H$. It follows that the restriction  $\varphi:G \to A$, of $\psi$ to $G$, is injective on all edge groups of $(\mathcal{G},\Gamma)$. Since every virtually cyclic subgroup of $A$ is cyclic, we deduce that each edge group $G_e$ is cyclic, so for every $e \in E\Gamma$ there is $g_e \in G$ such that $\alpha_e(G_e)=\langle g_e \rangle \leqslant G$.  

It is a well-known fact that free abelian groups are fully residually $\Z$, so there exists a homomorphism $\eta:A \to \Z$ such that $\eta(\varphi(g_e)) \neq 0$, as long as $g_e \neq 1$ in $G$, for all $e \in E\Gamma$. Clearly, the homomorphism $\xi \coloneq \eta \circ \varphi:G \to \Z$ is injective on $\alpha_e(G_e) =\langle g_e \rangle$, for each $e \in E\Gamma$.
\end{proof}

\begin{prop}\label{prop:tubular+VRC->virt_free-by-cyclic} Let $H$ be a tubular group with (VRC). Then $H$ has a normal subgroup of finite index $G \n_f H$ such that $G \cong F_n \rtimes \Z$, where $F_n$ is the free group of  rank $n$, for some $n \in \N$.
\end{prop}

\begin{proof} Since $H$ has (VRC) and the edge groups are cyclic, 
\Cref{lem:VRC->homom_to_Z} applies to provide us with 
$G \n_f H$ and $\xi:G \to \Z$, as in its statement. Clearly, $G$ is again a tubular group, being a finite index subgroup of $H$. Since $\xi$ is injective on each edge group of $G$, we can apply a result of Button \cite[Proposition~2.1]{Button-free-by-cyclic} to conclude that $G$ has a finitely generated free subgroup $F \n G$ such that $G/F \cong \Z$.    
\end{proof}

\begin{cor}\label{cor:G_k_that_are_virt_free-by-Z} Let $G_k$ be the group defined by \eqref{eq:pres_of_G_k}, for $k \in \Z$. Then $G_k$ has a finite index subgroup decomposing as $F_n \rtimes \Z$ if and only if $|k| \le 2$.    
\end{cor}

\begin{proof}
Propositions~\ref{prop:tubular+VRC->virt_free-by-cyclic} and \ref{prop:G_k}    tell us that for any $k \in \{0,\pm 1\}$ the group $G_k$ is virtually (finitely generated free)-by-cyclic. 

Let $D_\infty$ denote the infinite dihedral group, given by the presentation
\[D_\infty \coloneq \langle \delta, \tau \mid \delta^2=1,~\delta \tau \delta^{-1}=\tau^{-1} \rangle.\]
If $|k|=2$ then there is a homomorphism $\varphi:G_k \to D_\infty$ that is injective on the associated subgroups $\langle a \rangle$, $\langle b \rangle$ and $\langle ab^k \rangle$. 
Indeed, when $k=2$ we can define $\varphi(a) \coloneq \tau$, $\varphi(b) \coloneq \tau^{-1}$, $\varphi(s)\coloneq \delta$ and $\varphi(t) \coloneq 1$, so that \[\varphi(sas^{-1})=\tau^{-1}=\varphi(b)~\text{ and }~ \varphi(tbt^{-1})= \tau^{-1}=\varphi(ab^2).\]
And if $k=-2$, then we define
 $\varphi(a) \coloneq \tau$, $\varphi(b) \coloneq \tau^{-1}$ and $\varphi(s)=\varphi(t) \coloneq \delta$.

This implies that the associated subgroup of the HNN-extension $G_k$ are all virtual retracts, so the argument from the proof of \Cref{prop:tubular+VRC->virt_free-by-cyclic} applies to show that $G_k$ is virtually (finitely generated free)-by-cyclic, when $|k|=2$.

Finally, suppose that $|k| \ge 3$. If $K \leqslant_f G_k$ decomposes as a semidirect product $N \rtimes \Z$, then, by \Cref{lem:virt_retract->homom}, there is a finitely generated virtually abelian group $P$ and a homomorphism $\varphi:G_k \to P$ such that $\ker\varphi \subseteq N$. \Cref{rem:homoms_from_G_k} then implies that $a^m,b^m \in \ker\varphi$, for some $m \in \N$. Therefore, $N$ will contains the free abelian subgroup $\langle a^m,b^m \rangle \cong \Z^2$, so it cannot be free. Thus $G_k$ is not virtually free-by-cyclic, when $|k|\ge 3$.  
\end{proof}

\begin{ex}\label{ex:which_G_kl_are_free-by-Z}
Similarly to \Cref{cor:G_k_that_are_virt_free-by-Z}, using \Cref{ex:G_kl} one can show that the group $G_{k,l}$, defined by  \eqref{eq:G_kl}, is virtually (finitely generated free)-by-cyclic if and only if either $|k|=|l|=1$ or $|k \pm l|=1$ (as usual, we assume $(k,l) \neq (0,0)$).
\end{ex}

\begin{rem}\label{rem:virt_matters}
The tubular group $G_{1}$, given by \eqref{eq:pres_of_G_k}, has (VRC) by \Cref{prop:G_k}. But this group is not free-by-cyclic because the free abelian subgroup $\langle a, b \rangle \cong \Z^2$ is contained in the kernel of any homomorphism from $G_{1}$ to an abelian group. Thus,   in the statement of \Cref{prop:tubular+VRC->virt_free-by-cyclic}, the word ``virtually'' cannot be dropped.
\end{rem}

Gersten's example \cite{Gersten}, mentioned in the Introduction, shows that not every group of the form $F_n \rtimes \Z$ has (VRC).

\begin{problem}  Classify which (finitely generated free)-by-cyclic groups have (VRC) and which ones have (LR).    
\end{problem}

If $G=F_n \rtimes_\varphi \Z$ has (LR) then the image of $\varphi$ must have finite order in $\mathrm{Out}(F_n)$ by \Cref{lem:nvr->finite_image}, but we do not know if this condition is also sufficient for (LR) (it does imply that $G$ has (VRC), by \Cref{lem: listresultsretr}, because it has a finite index subgroup isomorphic to $F_n \times \Z$).

\appendix
\section{Tubular groups with (VRC) are virtually special}
\label{sec:appendix}

\ifnum \anonym=0
\smallskip
\begin{center}by \textsc{Jon Merladet Urig\"uen, Ashot Minasyan, Xiaolei Wu and Shengkui Ye}\end{center}
\medskip
\fi

In this appendix we prove \Cref{prop:tubular_+VRC->virt_spec} mentioned in the Introduction.

\begin{lemma}\label{lem:VRC->free_action_on_cube_complex}
Suppose that $H \in \mathfrak{C}$. If $H$ has (VRC) then there is a finite index subgroup $G \n_f H$ admitting  a free action on a finite dimensional CAT($0$) cube complex $\mathcal{X}$, which is the product of the Bass-Serre tree for $H$ with the standard cubulation of the Euclidean $n$-space, for some $n \in \N_0$.
\end{lemma}

\begin{proof} According to \Cref{prop:further_props_of_VRC}, there exist a finitely generated virtually abelian group $P$ and a homomorphism $\psi:H \to P$, such that $\psi$ is injective on each vertex group of $H$. Let $A\cong \Z^n$ be a free abelian normal subgroup of finite index in $P$, for some $n \in \N_0$, and set $G \coloneq \psi^{-1}(A) \n_f H$. 

Then $G$ acts on the Bass-Serre tree $\mathcal{T}$ of $H$, and $\psi$ is injective on the vertex stabilizers. Since $A \cong \Z^n$ acts freely by translations on the standard cubulation $\mathcal{E}$ of $\R^n$,  we obtain a diagonal action of $G$ on the cube complex $\mathcal{X} \coloneq \mathcal{T} \times \mathcal{E}$. Evidently, $\mathcal{X}$ is CAT($0$) and the action of $G$ on it is free because the $G$-stabilizers of vertices in $\mathcal{T}$ act freely on $\mathcal{E}$ by construction.
\end{proof}

In \cite[Theorem~1.1]{Woodhouse} Woodhouse proved that if a tubular group admits a free action on a finite dimensional CAT($0$) cube complex then it is virtually special. We can combine this result with \Cref{lem:VRC->free_action_on_cube_complex} to deduce
\Cref{prop:tubular_+VRC->virt_spec} from the Introduction. 

Recall that a group is said to be \emph{virtually compact special} if it has a finite index subgroup admitting an action on a CAT($0$) cube complex such that the quotient is a \emph{compact} special cube complex, in the sense of Haglund and Wise \cite{Haglund-Wise}.

\begin{ex} By \cite[Corollaries~5.10,5.9]{Wise-tubular} (see also \cite[Lemma~4.7]{Wu-Ye}) a tubular group $G$ is virtually \emph{compact} special if and only if $G$ is balanced and there are at most two parallelism classes of edges at each vertex. 

Combining this with \Cref{prop:tubular_+VRC->virt_spec} and \Cref{ex:G_kl}, we see that, given a pair of integers $(k,l) \in \Z^2 \setminus\{(0,0)\}$, the group $G_{k,l}$, defined by \eqref{eq:G_kl},  is virtually special if and only if $k,l \in \{0, \pm 1\}$, and it is virtually compact special if and only if $(k,l)\in \{(0,\pm 1),(\pm 1, 0)\}$.    
\end{ex}

\begin{question} Suppose that a group $G \in \mathfrak{C}$
has (VRC). Is $G$ necessarily virtually special?    
\end{question}


\printbibliography 
\end{document}